\documentclass{article}
\usepackage[latin1]{inputenc}
\usepackage[T1]{fontenc}
\usepackage{amsmath,amssymb,amstext}
\usepackage{theorem}
\usepackage{multicol}
\usepackage[english]{babel}
\usepackage{datetime}
\usepackage{enumerate}
\usepackage{eufrak}
\usepackage{graphicx}
\usepackage{caption}
\usepackage{subcaption}
\usepackage{float}
\usepackage{multirow}
\usepackage{color}

\graphicspath{ {./images/} }

\newtheorem{theorem}{Theorem}

\newtheorem{corollary}[theorem]{Corollary}

\newtheorem{example}[theorem]{Example}

\newtheorem{lemma}[theorem]{Lemma}

\newtheorem{proposition}[theorem]{Proposition}

\newenvironment{proof}[1][Proof]{\textbf{#1.} }{\ \rule{0.5em}{0.5em}}

\renewcommand{\geq}{\geqslant}
\def\leq{\leqslant}

\newcommand{\N}{\mathbb{N}}

\def\1{{\mathbf{1}}}

\def\1{{\mathbf{1}}}
\def\0.5{{\frac{1}{2}}}

\def \N{\mathbb{N}}

\newcommand{\HC}{\mathcal{H}}
\newcommand{\inner}[2]{\langle #1, #2 \rangle}

\def \ko{{\delta^{2 \alpha}_n}}

\begin{document}

\title{\textbf{Parameter estimation for fractional stochastic heat equations : Berry-Ess\'een bounds in CLTs}}
\author{
	Soukaina Douissi \thanks{%
		Ecole Nationale des Sciences Appliquées, BP 575, Avenue Abdelkrim Khattabi, 40000, Guéliz-Marrakech Email:\texttt{%
			s.douissi@uca.ma }} and Fatimah Alshahrani \thanks{Department of Mathematical Sciences, College of Science, Princess Nourah bint Abdulrahman University,
		P.O. Box 84428, Riyadh 11671, Saudi Arabia. Email:\texttt{fmalshahrani@pnu.edu.sa}}
}
\date{\today}
\maketitle

\noindent \textbf{Abstract:} The aim of this work is to estimate the drift coefficient of a fractional
heat equation driven by an additive space-time noise using the Maximum likelihood estimator (MLE). In the first part of the paper, the  first $N$ Fourier modes of the
solution are observed continuously over a finite time interval $[0, T ]$. The
explicit upper bounds for the Wasserstein distance for the central limit
theorem of the MLE  is provided when $N \rightarrow \infty$ and/or $T
\rightarrow \infty$. While in the second part of the paper, the $N$ Fourier modes are observed at uniform time grid : $t_i = i \frac{T}{M}$, $i=0,..,M,$ where $M$ is the number of time grid points.  The consistency and asymptotic normality are studied when $T,M,N \rightarrow + \infty$ in addition to the rate of convergence in law in the CLT. 
\section{Introduction}

The aim of this study is to solve the drift parameter estimation problem $\theta > 0$ of the following fractional stochastic heat equation :
\begin{equation}\label{SPDE}
\left\{
\begin{array}{ll}
dY_\theta(t,x) = - \theta \left( - \Delta\right)^{\alpha} Y_{\theta}(t,x) dt + \sum \limits_{n \in \mathbb{N}}
\delta^{-\gamma}_n e_n(x) dw_n(t), \text{ \ } 0 \leq t \leq T, \text{ \ } x \in G. &  \\
~~ &  \\
Y_{\theta}(0,x) = Y_0 \text{ \ } x \in G,
\end{array}%
\right.
\end{equation}
where $T>0$, $\theta >0$, $\alpha > 0$, $\gamma \geq 0$, G is a bounded and smooth domain in $\mathbb{R}^{d}$, $d \geq 1$. $\Delta$ is the Laplace operator on G with
zero boundary conditions and $H^{r}(G)$ for $r \in \mathbb{R} $, denotes the corresponding Sobolev spaces. The initial condition of (\ref{SPDE}) is such that $Y_0 \in H^{r}(G)$ for some $r \in \mathbb{R}$.
The process $\left\{ Y_{\theta}(t,x), x \in G, t  \in [0,T]
\right\}  $ is defined on a filtered probability space $ \left(\Omega,
\mathcal{F}, \left\{ \mathcal{F}_t\right\}_{t \geq 0} , \mathbf{P}
\right)$ on which we consider a family $\left\{ w_j, j \geq 1 \right\}$  of independent standard Brownian motions. \\
The set $\left\{ e_n, n \in \N \right\}$ are the eigenfunctions of $\Delta$ that form a complete orthonormal system in $L^2(G)$, the corresponding
eigenvalues $h_n, n \in \N$ can be sorted such that
$0< -h_1 \leq - h_2 \leq \ldots .$ \\ Let us introduce the notation  $\delta_n:= (-h_n)^{1/2}$, $n \in \N$, so, there exists a positive constant $\bar{\sigma}$ such that, see \cite{Sh}
\begin{equation}\label{vk}
\lim \limits_{n \rightarrow +\infty} |\delta_n|^2 n^{-2/d} = \bar{\sigma}.
\end{equation}
Assuming that $2(\gamma-r)> d$, then (see for instance \cite{chow, LR17, LR18}) equation (\ref{SPDE}) has a unique solution $Y_\theta$ weak in the PDE sense and strong in the probability sense.\\
In the following, it is assumed that $r \geq 0$ and $2 \gamma > d$.
Let $y_n$, $n \geq 1$ be the Fourier coefficient of the solution $Y_{\theta}$ of (\ref{SPDE}) with respect to $e_n$, $n \geq 1$, i.e.
$y_n(t) = (Y_{\theta}(t), e_n)_0$, $n \in \N$. Let $H^{N}$ be the finite dimensional subspace of $L^2(G)$ generated by $\left\{ e_n, n =1,...,N \right\}$
and denote $P_N$ the projection operator of $L^{2}(G)$ into $H^N$ and put ${Y_{\theta}}^N = P_N Y_{\theta}$. Each Fourier mode $y_n$, $n \geq 1$ follows the dynamics
of an Ornstein-Uhlenbeck process given by
\begin{equation}\label{OU}
dy_n(t) = - \theta \delta^{2 \alpha}_n y_n(t) dt + \delta^{- \gamma}_n dw_n(t),  \text{ \ }y_n(0) = (Y_0, e_n), \text{ \ } t \geq 0.
\end{equation}
In the following, we will denote by $\mathbf{P}^{\theta}_{T,N}$ the probability measure on $C([0,T];H^N)$ generated by ${Y_{\theta}}^N$. We will also fix in the rest of the paper, a parameter value $\theta_0>0$. Then since the measures $\left\{{\mathbf{P}^{\theta}_{T, N}}, \theta>0  \right\}$ are equivalent, hence applying Girsanov's theorem we obtain the following Likelihood Ratio or Radon-Nikodym derivative:
\begin{align*}
\frac{d{\mathbf{P}^{\theta}_{T,N}}}{{d\mathbf{P}^{\theta_0}_{T,N}}}({Y_{\theta}}^N) = \exp\left(
- (\theta-\theta_0) \sum\limits_{n=1}^{N} \delta^{2 \alpha +  \gamma}_n \int_{0}^{T} y_n(t) dy_n(t) - \frac{(\theta^2 - {\theta^2_0})}{2} \sum\limits_{n=1}^{N} \delta^{4 \alpha + 2 \gamma}_n
\int_{0}^{T} y^2_n(t) dt\right).
\end{align*}
Maximizing the log-likelihood ratio with respect to $\theta$ gives the Maximum Likelihood Estimator (MLE) for $\theta$ 
\begin{equation}\label{MLE}
\hat{\theta}_{T,N} := - \frac{\sum\limits_{n=1}^{N} \delta^{2 \alpha + 2 \gamma}_n \int_{0}^{T} y_n(t) dy_n(t) }{\sum\limits_{n=1}^{N} \delta^{4 \alpha + 2 \gamma}_n
	\int_{0}^{T} y^2_n(t) dt}, \text{ \ } N \in \N, \text{ \ } T >0.
\end{equation}
Moreover, using (\ref{OU}), we get :
%\begin{equation}\label{thetaMLE}
$\theta - \hat{\theta}_{T,N} =  \frac{\sum\limits_{n=1}^{N}
	\delta^{2 \alpha +  \gamma}_n \int_{0}^{T} y_n(t) dw_n(t)
}{\sum\limits_{n=1}^{N} \delta^{4 \alpha + 2 \gamma}_n
	\int_{0}^{T} y^2_n(t) dt}, \text{ \ } N \in \N, \text{ \ } T >0.$\\
%\end{equation}
The consistency
and asymptotic normality of  the MLE $\hat{\theta}_{T,N}$ have been
studied in several papers when $N \rightarrow \infty$ and/or $T
\rightarrow \infty$, see for instance \cite{cialenco,CX} and  \cite{CVK}.
However, to utilize the asymptotic distribution
of an estimator it is essential that the rate of convergence is known. For the best of our knowledge, no result
of the Berry-Ess\'een type is known for the distribution of the MLE
$\hat{\theta}_{T,N}$ of the drift parameter $\theta$ of the SPDE
\eqref{SPDE}.
\\
The aim of this work is to study the rate of convergence for the
central limit theorem of the MLE of
$\theta$ in its continuous and discrete versions. The article is
presented as follows: Section 2 consist of some of the essential tools needed from the analysis on Wiener space and Malliavin calculus.
\\
In Section 3, we provide explicit bounds of the convergence of the MLE to a Gaussian random variable for the Wasserstein distance,
considering three different scenarios: when $N \rightarrow +\infty$ and $T$ is fixed, when $T \rightarrow +\infty$ and $N$
is fixed and when both $N,T \rightarrow + \infty$, see Theorem \ref{main-thm} and Corollary \ref{cor. of main-thm}.
The key of the proof  is the fact that the numerator of $\theta_0- \hat{\theta}_{T,N}$ 
is a second chaos random sequence that depends on both $N$ and $T$.
A rigorous study of the third and fourth cumulant of the numerator along with the application of the Optimal
fourth moment theorem \cite{NP} lead to a CLT with the rates of its convergence in law. It is also worth noting
that the proof of the consistency part relies mainly on the properties of multiple Wiener integrals in conjunction with Borelli-Catelli Lemma.
\\
In Section 4, we study the asymptotic properties of an approximate version of
the MLE $\hat{\theta}_{T,N}$, which we denote by $\tilde{\theta}_{T, N, M}$, defined by
%\begin{equation*}
$\widetilde{\theta}_{T,N, M}:=-\frac{\sum_{n=1}^N \delta_n^{2
		\alpha+2 \gamma} \sum_{i=1}^M
	y_n\left(t_{i-1}\right)\left[y_n\left(t_i\right)-y_n\left(t_{i-1}\right)\right]}{\Delta_M\sum_{n=1}^N
	\delta_n^{4 \alpha+2 \gamma} \sum_{i=1}^M
	y_n^2\left(t_{i-1}\right)},$
%\end{equation*}
where we simply discretized the numerator and denominator of $\hat{\theta}_{N,T}$, considering
that the first $N$ Fourier modes are now observed at discrete time
instant $t_i = i \frac{T}{M}$, $i=0,..,M$ where $M$ denotes   the number of time grid points.
We proved that when $\frac{T}{M} N^{2\alpha/d} \rightarrow 0 $ as $T,M,N \rightarrow + \infty$,
the estimator $\tilde{\theta}_{T,N,M}$ is consistent and when $\frac{T^{3/2} N^{\frac{3 \alpha}{d}+ \frac{1}{2}}}{M} \rightarrow 0$
as ${T,M,N \rightarrow +\infty}$, then $\tilde{\theta}_{T,N,M}$ is asymptotically Gaussian. We also derive the rate of the
convergence in law of  $\tilde{\theta}_{T,N,M}$ for the Wasserstein distance, see Theorem \ref{thm-theta-tilde} for more details.

\section{Elements of analysis on Wiener space.}
In this section, we give a concise overview of some elements from Malliavin calculus and elements of Gaussian analysis. For additional details about this topic, the interested reader is referred to the following references  \cite{NP-book}, and \cite{nualart-book}.\\
Consider $W$ a Brownian motion defined on a probability space $(\Omega, \mathcal{F}, \mathbf{P})$. The Wiener integral of a deterministic function $g \in L^{2}([0,T])$, $\int_{0}^{T} g(s)dW(s)$ can also be denoted by $W(g)$. The Hilbert space $\mathcal{H} := L^{2}([0,T])$ is endowed with the inner product $$\mathbf{E}[W(f)W(g)]= \int_{0}^{T} f(s)g(s) ds = \inner f g_{L^{2}([0,T])}  $$ %$$
In the following, we will consider that $\mathcal{F}$ is generated by $W$ and $L^{2}(\Omega) := L^{2}(\Omega,\mathcal{F}, \mathbf{P}).$\\
For all $p \geq 1$, $\mathcal{H}^{W}_p$ denotes the p-th Wiener chaos of $W$ is the closed linear subspace of $L^{2}(\Omega)$ generated by $\left\{ H_p(W(g)), g \in L^{2}([0,T]), \| g\|_{L^{2}([0,T])} =1 \right\}$, where $H_p$ is the p-th Hermite polynomial defined by : 
\begin{equation*}
H_{p}(x)= (-1)^{p} e^{\frac{x^{2}}{2}} \frac{d^{p}}{dx^{p}} e^{-\frac{x^{2}}{2}}, \text{ \ } p \geq 1,
\end{equation*}
and $H_{0}(x) =1$. \\ 
In the rest of this section, the notations $\mathcal{H}^{\otimes p}$ and $\mathcal{H}^{\odot p}$ for any integer $p \geq 1$, mean the $p$-th tensor product and the $p$-th symmetric tensor product of $\mathcal{H}$ respectively.\\
The linear mapping defined by  $I^{W}_p(g^{\otimes p}) : = p! H_p(W(g))$ is an isometry between $\mathcal{H}^{\odot p} = L^2_s([0,T]^p)$ equipped with the norm $\sqrt{p!} \| . \|_{\mathcal{H}^{\odot p}}$ and the
Wiener chaos of order $p$ under $L^{2}\left( \Omega \right) $'s
norm. For $h \in \mathcal{H}^{\otimes p}$, $I^{W}_p(h)$ is a multiple Wiener-It\^o integral of order $p$ with respect to $W$ and we can write 
$$I^{W}_p(h) = \int_{[0,T]^p} h(x_1,....,x_p) dW(x_1)...dW(x_p).$$
Multiple Wiener integrals have many properties, we recall the main ones that we will need in our analysis. We start with the \textbf{isometry property}, which states that if $h_1 \in \mathcal{H}^{\odot p} $ and $ h_2 \in \mathcal{H}^{\odot q}$, for $p,q \geq 1$. Then the following holds :
\begin{equation}\label{Isometry}
\mathbf{E}\left[ I^{W}_{p}(h_1)I^{W}_{q}(h_2)\right] =
\begin{cases}
& p! \times \big\langle h_1, h_2 \big\rangle_{L^2(\mathbb{R}^{p})} \qquad\text{if } p=q\\
& 0 \qquad\qquad\qquad\quad\quad \text{ if } p \ne q.
\end{cases}
\end{equation}%
The second property we will be using is the so-called \textbf{product formula} : Let $p,q \geq 1$. If $h_1 \in \HC^{\odot p}$ and $h_2 \in \HC^{\odot q}$ then
\begin{align}
\label{eq:product} I_p(h_1) I_q(h_2) = \sum_{l = 0}^{p \wedge q} l! {p \choose l} {q \choose l} I_{p + q -2l}(h_1 \widetilde{\otimes}_l h_2).
\end{align}
where $%
h_1\otimes _{l}h_2$  is the contraction of $h_1$ and $h_2$ of order $l$
which is an element of ${\mathcal{H}}^{\otimes (p+q-2l)}$ defined by
\begin{eqnarray*}
	&&(h_1\otimes _{l}h_2)(r_{1},\ldots ,r_{p-l},s_{1},\ldots ,s_{q-l}) \\
	&&\qquad :=\int_{[0,T]^{p+q-2l}}h_1(r_{1},\ldots ,r_{p-l},v_{1},\ldots
	,v_{l})h_2(s_{1},\ldots ,s_{q-l},v_{1},\ldots ,v_{l})\,dv_{1}\cdots
	dv_{l}.
\end{eqnarray*}
The notation  $h_1\widetilde\otimes_l h_2$ in (\ref{eq:product}) means \textbf{the symmetrization} of $%
h_1\otimes _{l}h_2$  which is defined for a function $h$ by
$$\tilde{h}(y_{1},\ldots,y_{q}) = \frac{1}{q!} \sum\limits_{\sigma} h(y_{\sigma(1)},...,y_{\sigma(q)})$$
where all permutations $\sigma$ of $ \{1,...,q\}$ are included in the sum.
The special case $p=q=1$ in (\ref{eq:product}) is useful in practice can be written as follows:
\begin{equation}\label{product}
I_{1}(h_1)I_{1}(h_2)=2^{-1}I_{2}\left( h_1\otimes h_2+h_2\otimes h_1\right) +\langle
h_1,h_2\rangle _{{\mathcal{H}}}.
\end{equation}
Another important property in Wiener chaos is \textbf{the hypercontractivity} : If $q\geq1$, then for  any $p\geq2$, there exists $C(p,q)$
depending only on $p$ and $q$ such that, for every $X\in \oplus
_{l=1}^{q}{\mathcal{H}}_{l}$,
\begin{equation}
\left( \mathbf{E}\big[|X|^{p}\big]\right) ^{1/p}\leqslant C(p,q) \left( \mathbf{E}\big[|X|^{2}%
\big]\right) ^{1/2}.  \label{hypercontractivity}
\end{equation}%
The constants $C(p,q)$ above are known with some
precision when $F\in {\mathcal{H}}_{q}$; $C(p,q)=\left( p-1 \right) ^{q/2}$, see \cite{NP-book}.
We also recall the \textbf{Optimal fourth moment theorem} which will be used in order to derive rates of convergences in CLTs.\\
Let $N$ denote the standard normal random variable. Consider a random sequence $X_{n}\in {\mathcal{H}}_{q}$, $q \geq 1$ such that $Var\left[ X_{n}\right] =1$ and $\left( X_{n}\right)_{n}$ converges in distribution to $N$. The \emph{Fourth moment theorem} proved in~\cite{NuP} claims that this convergence is equivalent to $\lim_{n}\mathbf{E}\left[ X_{n}^{4}\right] =3$. The following optimal estimate for $d_{TV}\left( X_n,\mathcal{N}\right) $, known as the Optimal fourth moment theorem, was proved in \cite{NP}: with the sequence $X_n$ as above, assuming convergence, there exist two constants $c,C>0$ independent of $n$, such that
\begin{equation}
c\max \left\{ \mathbf{E}\left[ X_{n}^{4}\right] -3,\left\vert \mathbf{E}\left[ X_{n}^{3}\right] \right\vert \right\} \leqslant d_{TV}\left(X_{n},\mathcal{N}\right) \leqslant C\max \left\{ \mathbf{E}\left[X_{n}^{4}\right] -3,\left\vert \mathbf{E}\left[ X_{n}^{3}\right]\right\vert \right\} . \label{optimal berry esseen}
\end{equation}
Notice that optimal bound~\eqref{optimal berry esseen} holds with $d_{TV}$ replaced by $d_W$, see (\cite{DEKN}, Remark 2.2). We recall that, for two random variables $F$ and $G$, the former metrics are respectively given by
\begin{align}
d_{TV}\left( F,G\right) := \sup_{A\in \mathcal{B}({\mathbb{R}})}\left\vert \mathbf{P}\left[ F\in A\right] -\mathbf{P}\left[ G\in A\right] \right\vert, \ \  d_{W}\left( F,G\right) := \sup_{h\in Lip(1)}\left\vert \mathbf{E} [h(F)]-\mathbf{E} [h(G)]\right\vert,
\end{align}
where $Lip(1)$ is the set of all Lipschitz functions with Lipschitz constant $\leqslant 1.$ 
We also recall, the definition of the cumulant of order $p \geq 2$, $k_{p}(F)$ of a random variable $F= I_2(h)$, 
$$k_p(F) = 2^{p-1} (p-1)! \inner{ \ \underbrace{h {\otimes_1} \ldots {\otimes_1}h}_{p-1 \ copies \ of \ h}}{h}_{\HC^{\otimes 2} \ }.$$ 
In particular for $p=3,4$, we obtain estimates of the third and fourth cumulant respectively as follows :
\begin{eqnarray} 
k_{3}(F)  = \mathbf{E}[F^3] =8\langle h,h\otimes_1
h\rangle_{\mathcal{H}^{\otimes 2}}\label{3rd-cumulant}\  \text{and} \ 
k_{4}(F)
&=&  \mathbf{E}[F^4] -3\mathbf{E}[F^2]^2 \\
&=& 16\left(\| h \otimes_1 h\|_{\mathcal{H}^{\otimes
		2}}^2+2\|h\widetilde{\otimes_1} h\|_{\mathcal{H}^{\otimes
		2}}^2\right)\nonumber\\&\leq&48\|h\otimes_1
h\|_{\mathcal{H}^{\otimes 2}}^2.\label{4th-cumulant}
\end{eqnarray}
%For more details about cumulants and moments, the interested reader may consult the book \cite{NP-book}, section 8.
\section{Berry-Ess\'een bounds for the MLE when $N$ and $T$ are large: Continuous time observations}
For simplicity, we set $Y_0 = 0$ hence $y_n(0)= 0$ for all $n \geq
1$. Because equation \eqref{OU}
is linear, it is immediate to see that its solution can be expressed
explicitly as follows :
$$y_n(t) = \delta_n^{-\gamma} \int_{0}^{t} e^{- \theta_0 \delta_n^{2 \alpha} (t-u)} dw_n(u), \text{ \ } n = 1,...,N.$$
Define:
\begin{equation}\label{ENT}
E_{T,N} := \sum\limits_{n=1}^{N} \delta^{2 \alpha+ \gamma}_n \int_{0}^{T} y_n(t) dw_n(t), \ \ \psi^{\theta}_{T,N} := \frac{T}{2 \theta_0} \sum \limits_{n=1}^{N} \delta^{2 \alpha}_{n} \sim \frac{\bar{\sigma} d T N^{\frac{2 \alpha}{d}+1}}{(4 \alpha + 2d) \theta_0} \text{ \ } \text{as} \text{ \ } N,T \rightarrow +\infty.
\end{equation}
Indeed, by (\ref{vk}), we have $\delta_n \sim \bar{\sigma}^{1/2} n^{1/d}$ as $n \rightarrow +\infty$, thus $\sum\limits_{n=1}^{N} \delta^{2 \alpha}_n \sim \bar{\sigma}^{\alpha} \frac{N^{\frac{2 \alpha}{d}+1}}{\frac{2 \alpha}{d} +1},$ as $N \rightarrow +\infty$. \\
By (\ref{ENT}), we can write
\begin{align}\label{MLE-exp}
\theta_0 - \hat{\theta}_{T,N} & =  \frac{\sum\limits_{n=1}^{N} \delta^{2 \alpha +  \gamma}_n \int_{0}^{T} y_n(t) dw_n(t) }{\sum\limits_{n=1}^{N} \delta^{4 \alpha \ + 2 \gamma}_n
	\int_{0}^{T} y^2_n(t) dt} =\frac{E_{T,N}}{< E_{T,N} >}.
\end{align}
The numerator term of (\ref{MLE-exp}) is a second chaos random variable, indeed for every $ n \geq 1,$
\begin{align*}
\int_{0}^{T} y_{n}(t) dw_n(t) & = \delta^{- \gamma}_n \int_{0}^{T} \int_{0}^{t} e^{- \theta_0 \delta^{2 \alpha}_n(t-u)} dw_n(u) dw_n(t) = I^{w_n}_2(g^{n}),
\end{align*}
where $g^n(s,t) := \delta^{2 \alpha}_{n} e^{- \theta_0 \delta^{2\alpha}_n |t-s| } \mathbf{1}_{[0,T]^2}(t,s), \text{ \ } n=1,...,N.$
Therefore, using the linear property of multiple integrals, we can express $E_{T,N}$ as follows
\begin{equation}\label{ETN}
E_{T,N} = I_2\left(\frac{g_{T,N}}{2}\right) = \frac{1}{2} \sum\limits_{n=1}^{N} I^{w_n}_2(g^{n})
\end{equation}
where $g_{T,N} = (g^1,...,g^N)$. For the denominator term, using the product formula (\ref{product}), we get for every $n=1,...,N$ 
\begin{align*}
y^2_n(t) & = \delta^{-2 \gamma}_{n}  e^{- 2 \theta_0 \delta_n^{2
		\alpha}t}  I^{w_n}_2 \left(e^{\theta_0 \delta^{2 \alpha}_n(u+v)}
\mathbf{1}(u,v)_{[0,t]^2} \right) + \frac{\delta^{-2 \gamma}_n}{2
	\theta_0 \ko} \left(1-e^{-2 \theta_0 \ko t} \right).
\end{align*}
Integrating over $[0,T]$, 
\begin{align*}
\int_{0}^{T} y^2_n(t) dt & = \delta^{-2 \gamma}_{n} I^{w_n}_2\left( e^{\theta_0 \delta^{2 \alpha}_n (u+v)} \int_{u \vee v}^{T} e^{- 2 \theta_0 \delta^{2 \alpha}_n s} ds\right) + \frac{\delta^{-2 \gamma}_{n}}{2 \theta_0 \delta^{2 \alpha}_n} \left(T + \frac{e^{-2 \theta_0 \delta^{2 \alpha}_n T} -1}{2 \theta_0 \delta^{2 \alpha}_n} \right) \\
& =  \delta^{-2 \gamma}_{n} I^{w_n}_2\left( \frac{e^{- \theta_0 \delta^{2 \alpha}_n |u-v|}}{2 \theta_0 \delta^{2 \alpha}_n} \mathbf{1}(u,v)_{[0,T]^2} \right) - \delta^{-2 \gamma}_{n} I^{w_n}_2\left( \frac{e^{-2 \theta_0 \delta^{2 \alpha}_n T} e^{\theta_0 \delta^{2 \alpha}_n (u+v)}}{2 \theta_0 \delta^{2 \alpha}_n} \mathbf{1}(u,v)_{[0,T]^2} \right) \\
& \hspace*{1cm} + \frac{\delta^{-2 \gamma}_{n}}{2 \theta_0 \delta^{2 \alpha}_n} \left( T + \frac{e^{-2 \theta_0 \delta^{2 \alpha}_n T} -1}{2 \theta_0 \delta^{2 \alpha}_n}\right)
\end{align*}
Thus
\begin{align}\label{crochet-ENT}
\langle E_{T,N}\rangle& = \sum \limits_{n=1}^{N} \delta^{4 \alpha +2 \gamma}_n \int_{0}^{T} y^2_n(t) dt   = \sum\limits_{n=1}^{N} \ko I^{w_n}_2\left( \frac{e^{- \theta_0  |u-v|}}{2 \theta_0 } \mathbf{1}(u,v)_{[0,T]^2} \right) - \sum\limits_{n=1}^{N} \ko I^{w_n}_2\left( \frac{e^{-2 \theta_0 \ko T} e^{\theta_0 \ko (u+v)}}{2 \theta_0 } \mathbf{1}(u,v)_{[0,T]^2} \right) \notag \\
& \hspace*{5cm} +  \sum\limits_{n=1}^{N} \frac{\delta^{2 \alpha}_{n}}{2 \theta_0 } \left( T + \frac{e^{-2 \theta_0 \ko T} -1}{2 \theta_0 \ko}\right) \notag \\
&\hspace*{3.5cm}  = I_2(\frac{g_{T,N}}{2 \theta_0}) - I_2(l_{T,N}) + \lambda_{T,N} = \frac{E_{T,N}}{\theta_0} - I_2(l_{T,N}) + \lambda_{T,N}  
\end{align}
with $ l_{T,N} :=(l^1,...,l^N) \text{\ } \text{where} \text{ \ }
l^{n}(u,v) := \frac{\ko}{2 \theta_0} e^{-2 \theta_0 \ko T} e^{\theta_0 \ko (u+v)}
\mathbf{1}(u,v)_{[0,T]^2}$
and
\begin{align*}
\lambda_{T,N} & = \frac{1}{2 \theta_0} \sum \limits_{n=1}^{N} \ko \left( T + \frac{e^{-2 \theta_0 \ko T}-1}{2 \theta_0 \ko }\right) = \frac{T}{2 \theta_0} \sum \limits_{n=1}^{N} \ko  + \frac{1}{4 \theta_0} \sum \limits_{n=1}^{N} (e^{-2 \theta_0 \ko T}-1) \\
& = \psi^{\theta_0}_{T,N} + \frac{1}{4 \theta_0} \sum \limits_{n=1}^{N} (e^{-2 \theta_0 \ko T}-1).
\end{align*}
By (\ref{MLE-exp}), (\ref{ENT}) and (\ref{crochet-ENT}), we can write, for every $N \geq 1$, $T >0$,
\begin{equation}\label{MLE-theta0}
\theta_0 - \hat{\theta}_{T,N} :=
\frac{\frac{1}{2}I_2(g_{T,N})}{\frac{1}{2 \theta_0} I_2(g_{T,N})-
	I_2(l_{T,N}) + \lambda_{T,N}}.
\end{equation}
Note that the strong consistency of $\hat{\theta}_{T,N}$  when $N
\rightarrow \infty$ and $T$ fixed or  $T \rightarrow \infty$ and $N$
fixed, and when both  $N, T \rightarrow \infty$
has been already proved in  \cite{cialenco,CX} and
\cite{CVK}, respectively; here we propose a different approach based
on the Borel-Cantelli's Lemma and the properties of multiple Wiener integrals recalled in the preliminaries.
\begin{theorem}\label{consistency-thm}
	Suppose that $\alpha > 0$, $\theta_0 >0$, $ \gamma \geq 0$, $d \geq 1$ in equation (\ref{SPDE}). Let $\hat{\theta}_{T,N}$ be the MLE given in (\ref{MLE}) and $\psi^{\theta}_{T,N}$ given in (\ref{ENT}). Then, for every $T> 0$ fixed (resp.  $N \geq 1$ fixed), we have 
		$$ \hat{\theta}_{T,N} \longrightarrow \theta_0, \text{ \ } \text{almost surely as} \text{ \ } N \rightarrow +\infty \  (\text{resp.} \ T \rightarrow +\infty).$$
\end{theorem}
\begin{proof}
Assume that $T> 0$ is fixed, we can write from (\ref{MLE-theta0}),
	\begin{equation}
	\theta_0 - \hat{\theta}_{T,N} := \frac{\frac{1}{2}I_2(g_{T,N})/
		\psi^{\theta_0}_{T,N}}{ \left(\frac{1}{2 \theta_0} I_2(g_{T,N})- I_2(l_{T,N}) +
		\lambda_{T,N} \right)/ \psi^{\theta_0}_{T,N}}.
	\end{equation}
	We have
	\begin{align}\label{numerator}
	\mathbf{E}\left[\left(\frac{1}{2}I_2(g_{T,N}) \right)^2\right] & = \frac{1}{2} \|g_{T,N} \|^2 = \frac{1}{2} \sum\limits_{n=1}^{N} \delta^{4 \alpha}_n \int_{0}^{T} \int_{0}^{T} e^{-2 \theta_0 \ko |u-v|} du dv  = \sum\limits_{n=1}^{N} \delta^{4 \alpha}_n \int_{0}^{T} \int_{0}^{u} e^{-2 \theta_0 \ko (u-v)} du dv \notag \\
	&\hspace{2cm} = \sum\limits_{n=1}^{N} \delta^{4 \alpha}_n   \int_{0}^{T} \int_{0}^{u} e^{- 2 \theta_0 \ko w} dw du  = \sum\limits_{n=1}^{N} \delta^{4 \alpha}_n \int_{0}^{T} \left( \frac{1 - e^{- 2 \theta_0 \ko u}}{2 \theta_0 \ko} du\right)\notag  \\
	&\hspace{2cm}  = \psi^{\theta_0}_{T,N} + \frac{1}{2 \theta_0 } \sum\limits_{n=1}^{N}  (e^{- 2 \theta_0 \ko T}-1)  = \lambda_{T,N}.
	\end{align}Thus, there exists a constant $C(\alpha, d, \theta_0 , T)$ such that
	\begin{align*}
	\mathbf{E}\left[\left(\frac{\frac{1}{2}I_2(g_{T,N})}{\psi^{\theta_0}_{T,N}}  \right)^2\right] & = \frac{1}{\psi^{\theta_0}_{T,N}} + \frac{1}{(\psi^{\theta_0}_{T,N})^2} \frac{1}{2 \theta_0} \sum \limits_{n=1}^{N}  (e^{- 2 \theta_0 \ko T}-1)  \leq \frac{1}{\psi^{\theta_0}_{T,N}} + \frac{N}{2 \theta_0  (\psi^{\theta_0}_{T,N})^2} \leq \frac{C(\alpha, \theta_0, d, T)}{N^{\frac{2 \alpha}{d} +1}}.
	\end{align*}
	It follows that for every $\varepsilon > 0$, we have by Markov's inequality
	\begin{align*}
	\sum\limits_{N=1}^{+ \infty} \mathbf{P}\left(\left|\frac{\frac{1}{2}I_2(g_{T,N})}{\psi^{\theta_0}_{T,N}} \right|> \varepsilon \right) & \leq \frac{1}{\varepsilon^2} \sum\limits_{N=1}^{+ \infty} \mathbf{E}\left[\left(\frac{\frac{1}{2}I_2(g_{T,N})}{\psi^{\theta_0}_{T,N}} \right)^2 \right]  \leq \frac{C(\alpha, \theta_0, d, T)}{\varepsilon^2} \sum\limits_{N=1}^{+ \infty} \frac{1}{N^{\frac{2 \alpha}{d}+1}} < + \infty.
	\end{align*}
	Hence, from Borel-Cantelli's Lemma, for every $T > 0$, we have $\frac{\frac{1}{2}I_2(g_{T,N})}{\psi^{\theta_0}_{N,T}} \longrightarrow 0 \text{ \ }  \text{a.s} \text{ \ } \text{as} \text{ \ }  N \rightarrow +\infty.$
On the other hand,
	\begin{align*}
	\|l_{T,N} \|^2 &= \frac{1}{4 \theta_0^2} \sum \limits_{n=1}^{N} \delta^{4 \alpha}_n \int_{0}^{T} \int_{0}^{T} e^{2 \theta_0 \ko (u+v)} du dv  = \frac{1}{16 \theta^4_0} \sum\limits_{n=1}^{N} (1- e^{-2 \theta_0 \ko T})^2  \leq \frac{N}{16 \theta^4_0}.
	\end{align*}
	We obtain therefore
	\begin{equation*}
	\mathbf{E}\left[
	\left(\frac{I_2(l_{T,N})}{\psi^{\theta_0}_{T,N}}\right)^2\right] = \frac{2
		\|l_{T,N} \|^2}{(\psi^{\theta_0}_{T,N})^2} \leq C(\alpha,d, \theta_0,T)
	\frac{1}{N^{\frac{4 \alpha}{d} + 1}}
	\end{equation*}
	Hence, for every $\varepsilon > 0$, we have
	\begin{align*}
	\sum\limits_{N=1}^{+ \infty} \mathbf{P}\left(\left|\frac{I_2(l_{T,N})}{\psi^{\theta_0}_{T,N}} \right|> \varepsilon \right) \leq \frac{1}{\varepsilon^2} \sum\limits_{N=1}^{+ \infty} \mathbf{E}\left[\left(\frac{I_2(l_{T,N})}{\psi^{\theta_0}_{T,N}} \right)^2 \right] 
	 \leq \frac{C(\alpha, \theta_0, d, T)}{\varepsilon^2} \sum\limits_{N=1}^{+ \infty} \frac{1}{N^{\frac{4 \alpha}{d}+1}} < + \infty.
	\end{align*}
	Hence, from Borel-Cantelli's Lemma, for every $T > 0$, we have
	%\begin{equation*}
	$\frac{I_2(l_{T,N})}{\psi^{\theta_0}_{T,N}} \longrightarrow 0 \text{ \ }  \text{a.s} \text{ \ } \text{as} \text{ \ }  N \rightarrow +\infty.$
	%\end{equation*}
	For the last deterministic term of the denominator, we have
	\begin{align*}
	\frac{\lambda_{T,N}}{\psi^{\theta_0}_{T,N}} & = 1 + \frac{1}{4 \theta_0^2 \psi^{\theta_0}_{T,N}} \sum \limits_{n =1}^{N} \left( e^{-2 \theta_0 \ko T} -1\right) \ \longrightarrow 1  \text{ \ } \text{as} \text{\ } N \rightarrow + \infty.
	\end{align*}
	Therefore the desired result follows.\\
Assume now that $N \geq 1$ is fixed, we can write
	\begin{equation*}
	\theta_0 - \hat{\theta}_{T,N}  =  \frac{\sum\limits_{n=1}^{N}
		\delta^{2 \alpha +  \gamma}_n \int_{0}^{T} y_n(t) dw_n(t)}{T}
	\times \frac{T}{\sum\limits_{n=1}^{N} \delta^{4 \alpha + 2
			\gamma}_n\int_{0}^{T} y^2_n(t) dt}
	\end{equation*}
	We have for any fixed $N \geq 1$,
	\begin{itemize}
		\item $\frac{1}{T} \sum\limits_{n=1}^{N} \delta^{2 \alpha +  \gamma}_n \int_{0}^{T} y_n(t) dw_n(t) \rightarrow 0 $ a.s. as $T \rightarrow + \infty$.
		\item $\frac{1}{T} \sum\limits_{n=1}^{N} \delta^{4 \alpha + 2 \gamma}_n\int_{0}^{T} y^2_n(t) dt \rightarrow \frac{1}{2 \theta_0} \sum\limits_{n=1}^{N} \delta^{2 (\alpha + \gamma)}_n$  a.s. as $T \rightarrow + \infty$.
	\end{itemize}
	In fact, using Ito's formula gives for all $n \geq 1$,
	\begin{equation*}
	y^2_n(T) = -2 \theta_0 \delta^{2 \alpha}_n\int_{0}^{T} y^2_n(t) dt + 2 \delta^{- \gamma}_n \int_{0}^{T} y_n(t) dw_n(t) + T
	\end{equation*}
	Thus,
	\begin{equation*}
	\frac{\delta^{- \gamma}_n}{T} \int_{0}^{T}y_n(t)dw_n(t) = \frac{y^2_n(T)}{2 T}  + \frac{\theta_0 \delta^{2 \alpha}_n }{T} \int_{0}^{T} y^2_n(t) dt - \frac{1}{2}.
	\end{equation*}
	We set $z_n(t) := \int_{- \infty}^{t} e^{- \theta_0  \delta^{2 \alpha}_n (t-u)} dw_n(t)$
	which is Gaussian, stationary and ergodic, thus by the ergodic
	theorem $$\frac{1}{T} \int_{0}^{T} z^2_n(t) dt \rightarrow
	\mathbf{E}[z^2_n(0)] = \frac{1}{2 \theta_0 \delta^{2 \alpha}_n}$$ a.s. as $T \rightarrow
	+\infty$. Since $z_n(t) = y_n(t) + e^{- \theta_0 \delta^{2 \alpha}_n t } \varepsilon_n(0)$
	where $\varepsilon_n(0) := \int_{-\infty}^{0} e^{\theta_0 \delta^{2 \alpha}_n u} dw_n(u)$.
	It follows that for all $n \geq 1$,
	$$\frac{1}{T} \int_{0}^{T} y^2_n(t) dt \rightarrow \frac{1}{2 \theta_0 \delta^{2 \alpha}_n} $$
	a.s. as $T \rightarrow +\infty$. Moreover, it was proved in \cite{HNZ} in Lemma 6.7 that for any $n \geq 1$, we have for any $\alpha> 0$, $ \frac{z_n(T)}{T^\alpha} \rightarrow 0$ a.s. as $T \rightarrow +\infty$. It follows that $\frac{y^2_n(T)}{T} \rightarrow 0$ a.s. as $T \rightarrow +\infty$ for any $n \geq 1$. Therefore, we get for any $n \geq 1$
	\begin{itemize}
		\item $\frac{1}{T} \int_{0}^{T} y_n(t)dw_n(t) \rightarrow 0$ a.s. as $T \rightarrow +\infty$.
		\item $\frac{1}{T} \int_{0}^{T} y^2_n(t) dt \rightarrow \frac{1}{2 \theta_0 \delta^{2 \alpha}_n}$  a.s. as $T \rightarrow +\infty$.
	\end{itemize}
	Then the desired result follows.
\end{proof}
\begin{theorem}\label{main-thm}
	Suppose that $\alpha > 0$, $\theta_0 >0$, $ \gamma \geq 0$, $d \geq 1$
	in equation (\ref{SPDE}). Let $\hat{\theta}_{T,N}$ be the MLE given
	in (\ref{MLE}) and $\psi^{\theta_0}_{T,N}$ given in (\ref{ENT}). Then there
	exists a positive constant $C(\alpha, d, \theta_0)$ depending on
	$\alpha$ the power of the Laplacian, $d$ is the space dimension and
	$\theta_0$ such that for any $N \geq 1$ and any $T>0$,
	
	\begin{equation}
	d_{W} \left(\sqrt{\psi^{\theta_0}_{T,N}} (\theta_0- \hat{\theta}_{T,N}), Z \right)     \leq C(\alpha, d, \theta_0) \times %
	\left\{
	\begin{array}{ll}
	\frac{1}{\sqrt {T}} \frac{1}{N^{\frac{2 \alpha}{d}}} & \mbox{ if } 2\alpha < d \\
	&  \\
	\frac{1}{\sqrt{T}} \frac{1}{N} & \mbox{ if } 2 \alpha =d%
	\\
	~~ &  \\
	\frac{1}{\sqrt{T}} \frac{1}{N^{\frac{\alpha}{d}+ \frac{1}{2}}} & \mbox{ if }2 \alpha > d,\\
	~~ &
	\end{array}%
	\right.
	\end{equation}
	where $Z \sim \mathcal{N}(0,1)$. Consequently as $T \rightarrow +\infty$, $N \rightarrow +\infty$
	\begin{equation*}
	\sqrt{T} N^{\frac{\alpha}{d}+\frac{1}{2}}\left(\theta_0 - \hat{\theta}_{T,N}\right)\overset{law}{\longrightarrow }\mathcal{N}%
	\left(0, \frac{(4\alpha /d +2) \theta_0}{\bar{\sigma}^{\alpha}} \right).
	\end{equation*}
\end{theorem}
\begin{proof}
	We recall that from (\ref{MLE-exp}) and (\ref{MLE-theta0}), we have $\theta_0 - \hat{\theta}_{T,N}  = \frac{E_{T,N}}{< E_{T,N} >}= \frac{\frac{1}{2}I_2(g_{T,N})}{\frac{1}{2 \theta_0} I_2(g_{T,N})-
		I_2(l_{T,N}) + \lambda_{T,N}}$.\\
On the other hand using the calculus (\ref{numerator}), the following estimate holds
	\begin{equation*}
	\mathbf{E}\left[ \left(
	\frac{E_{T,N}}{\sqrt{\psi^{\theta_0}_{T,N}}}\right)^2\right] =
	\frac{\lambda_{N,T}}{\psi^{\theta_0}_{T,N}} = 1 + \frac{1}{4
		\theta^2_{0} \psi^{\theta_0}_{T,N}} \sum \limits_{n=1}^{N}(e^{- 2 \theta_0
		\ko T}-1).
	\end{equation*}
	Therefore,
	\begin{equation}\label{normENT}
	\left| \mathbf{E}\left[ \left(
	\frac{E_{T,N}}{\sqrt{\psi^{\theta_0}_{T,N}}}\right)^2\right] -1
	\right| = |\frac{\lambda_{N,T}}{\psi^{\theta_0}_{T,N}} - 1| \leq
	\frac{C(\alpha, d , \theta_0)}{T N^{\frac{2 \alpha}{d}}},
	\end{equation}
	where $C(\alpha, d, \theta)$ is a constant depending on $\alpha$, $d$
	and $\theta$. Moreover, since $\frac{E_{T,N}}{\sqrt{\psi^{\theta_0}_{T,N}}} \in
	\mathcal{H}_2$ and $\mathbf{E}\left[ \left(
	\frac{E_{T,N}}{\sqrt{\psi^{\theta_0}_{T,N}}}\right)^2\right] \leq 1$, then by
	the hypercontractivity property, we have
	%\begin{equation*}
	$\| \frac{E_{T,N}}{\sqrt{\psi^{\theta_0}_{T,N}}} \|_{L^4} \leq 3 \| \frac{E_{T,N}}{\sqrt{\psi^{\theta_0}_{T,N}}} \|_{L^2} \leq 3.$
	%\end{equation*}
	\begin{lemma}
		Consider $\langle E_{T,N} \rangle$ defined in (\ref{crochet-ENT}) and $\psi^{\theta_0}_{T,N}$ defined in (\ref{ENT}). Then, there exists a constant depending on $\alpha$, $d$ and $\theta_0$, $C(\alpha,d,\theta_0)$ such that for any $N \geq 1$ and $T >0$, we have 
		\begin{equation}\label{crochetENT}
		\| \frac{<E_{T,N}>}{\psi^{\theta_0}_{T,N}} -1 \|_{L^2}    \leq C(\alpha, d, \theta_0)  \times b_{N,T}, \ \text{where} \	b_{N,T}:=  %
		\left\{
		\begin{array}{ll}
		\frac{1}{\sqrt {T}} \frac{1}{N^{\frac{2 \alpha}{d}}} & \mbox{ if } 2\alpha < d \\
		&  \\
		\frac{1}{\sqrt{T}} \frac{1}{N} & \mbox{ if } 2 \alpha =d%
		\\
		~~ &  \\
		\frac{1}{\sqrt{T}} \frac{1}{N^{\frac{\alpha}{d}+ \frac{1}{2}}} & \mbox{ if }2 \alpha > d. \\
		~~ &
		\end{array}%
		\right.
		\end{equation}
	\end{lemma}
\begin{proof}
	Indeed, using (\ref{crochet-ENT}), (\ref{normENT}) and Minkowski's inequality, we can write 
	\begin{align*}
	\| \frac{<E_{T,N}>}{\psi^{\theta_0}_{T,N}} -1 \|_{L^2} & \leq  \|\frac{E_{T,N}}{\theta_0 \psi^{\theta_0}_{T,N}} \|_{L^2} + \|\frac{I_2(h_{N,T})}{\psi^{\theta_0}_{T,N}} \|_{L^{2}}+ |\frac{\lambda_{N,T}}{\psi^{\theta_0}_{T,N}}-1| \leq \frac{1}{\theta_0} \frac{1}{\sqrt{\psi^{\theta_0}_{T,N}}} +
	\frac{C(\alpha, d, \theta_0)}{T N^{\frac{2\alpha}{d}+ \frac{1}{2}}} +
	\frac{C(\alpha,d, \theta_0)}{T N^{\frac{2\alpha}{d}}}.
	\end{align*}
	where we used for the estimate of $\|\frac{I_2(l_{T,N})}{\psi^{\theta_0}_{T,N}} \|_{L^{2}}$, the fact that 
	\begin{align*}
	\|l_{T,N} \|^2 &= \frac{1}{4 \theta^2_0} \sum \limits_{n=1}^{N} \delta^{4 \alpha}_n \int_{0}^{T} \int_{0}^{T} e^{2 \theta_0 \ko (u+v)} du dv  = \frac{1}{16 \theta^4_0} \sum\limits_{n=1}^{N} (1- e^{-2 \theta_0 \ko T})^2 \leq \frac{N}{16 \theta^4_0}. 
	\end{align*}
	We obtain therefore
	\begin{equation*}
	\mathbf{E}\left[
	\left(\frac{I_2(l_{T,N})}{\psi^{\theta_0}_{T,N}}\right)^2\right] =
	\frac{2 \|l_{T,N} \|^2}{(\psi^{\theta_0}_{T,N})^2} \leq C(\alpha,d,
	\theta_0) \frac{1}{T^2} \frac{1}{N^{\frac{4 \alpha}{d} + 1}}.
	\end{equation*}
\end{proof}
\begin{lemma}
	Consider $E_{T,N}$, $\psi^{\theta_0}_{T,N}$ defined respectively in (\ref{ENT}), then the following estimates hold 
	\begin{equation*}
	\kappa_3\left( \frac{E_{T,N}}{\sqrt{\psi^{\theta_0}_{T,N}}}\right) \leq \frac{3}{\theta_0 \sqrt{\psi^{\theta_0}_{T,N}}}, \ \ \text{and} \  \ \left|\kappa_4\left( \frac{E_{T,N}}{\sqrt{\psi^{\theta_0}_{T,N}}}\right)\right| \leq \frac{18}{ \theta^2_0 \psi^{\theta_0}_{T,N}}
	\end{equation*}
\end{lemma}
\begin{proof}
	we have by (\ref{3rd-cumulant})
	\begin{align*}
	\kappa_3\left( \frac{E_{T,N}}{\sqrt{\psi^{\theta_0}_{T,N}}}\right)& = \kappa_3\left( \frac{I_2(g_{T,N})}{2 \sqrt{\psi^{\theta_0}_{T,N}}}\right) = \mathbf{E}\left[\left(\frac{I_2(g_{T,N})}{2 \sqrt{\psi^{\theta_0}_{T,N}}}\right)^3 \right] = \sum\limits_{n=1}^{N} \mathbf{E}\left[ \left(\frac{I_2(g^n)}{2 \sqrt{\psi^{\theta_0}_{T,N}}}\right)^3\right] \\
	& = \frac{1}{(\psi^{\theta_0}_{T,N})^{3/2}} \sum\limits_{n=1}^{N} <g^n,
	g^n \otimes_{1} g^n>_{\mathcal{H}^{\otimes 2}}.
	\end{align*}
	Moreover, we have for every $n= 1,...,N$
	\begin{align*}
	<g^n, g^n \otimes_{1} g^n>_{\mathcal{H}^{\otimes 2}} & =
	\int_{0}^{T} \int_{0}^{T} \int_{0}^{T} g^n(x,y) g^n(x,t) g^n(y,t) dt dx dy \\
	& = \delta^{6 \alpha}_{n} \int_{[0,T]^{3}} e^{- \theta_0 \ko |x-y|} e^{ - \theta_0 \ko |x-t|} e^{- \theta_0 \ko |y-t|} dt dx dy \\
	& =  3! \delta^{6 \alpha}_{n} \int_{0}^{T} \int_{0}^{y} \int_{0}^{t} e^{- \theta_0 \ko (y-x)} e^{ - \theta_0 \ko (t-x)} e^{ - \theta_0 \ko (t-y)} dt dx dy \\
	& =  3! \delta^{6 \alpha}_{n} \int_{0}^{T} \int_{0}^{y} \int_{0}^{t} e^{ - \theta_0 \ko (2t -2x)} dt dx dy \\
	&= \frac{3 \delta^{4 \alpha}_n}{\theta_0} \int_{0}^{T} \left( \frac{1-
		e^{- 2 \theta_0 \ko t }}{2 \theta_0 \ko} - t e^{ - 2 \theta_0 \ko t}\right) dt.
	\end{align*}
	Thus,
	\begin{equation*}
	\kappa_3\left( \frac{E_{T,N}}{\sqrt{\psi^{\theta_0}_{T,N}}}\right) \leq \frac{3 T}{2  \theta_0 (\psi^{\theta_0}_{T,N})^{3/2}} \sum \limits_{n=1}^{N} \ko = \frac{3}{\theta_0 \sqrt{\psi^{\theta_0}_{T,N}}}.
	\end{equation*}
	We also have by (\ref{4th-cumulant})
	\begin{align*}
	|\kappa_4\left( \frac{E_{T,N}}{\sqrt{\psi^{\theta_0}_{T,N}}}\right)| & = |\kappa_{4} \left(\frac{I_2(g_{T,N})}{2 \sqrt{\psi^{\theta_0}_{T,N}}}\right)| \\
	& \leq \frac{3}{(\psi^{\theta_0}_{T,N})^2} \sum\limits_{n=1}^{N} \| g^n \otimes_{1} g^n\|^2_{\mathcal{H}^{\otimes 2}} \\
	& = \frac{3}{(\psi^{\theta_0}_{T,N})^2} \sum\limits_{n=1}^{N} \int_{[0,T]^4} g^n(x_1,x_2) g^n(x_2,x_3) g^n(x_3,x_4) g^n(x_4,x_1)dx_1 dx_2 dx_3dx_4 \\
	& = \frac{3}{(\psi^{\theta_0}_{T,N})^2} \sum \limits_{n=1}^{N} 4! \delta^{8 \alpha}_n \int_{0}^{T}
	\int_{0}^{x_4} \int_{0}^{x_3} \int_{0}^{x_2} e^{- \theta_0 \ko (2x_4-2x_1)} dx_1 dx_2dx_3 dx_4 \\
	& \leq \frac{18}{ \theta^2_0 \psi^{\theta_0}_{T,N}}.
	\end{align*}
\end{proof}
Consequently, there exists a positive constant $C(\theta_0)$ depending only on $\theta_0$ such that for every $N \geq 1$, $T > 0$, we have
\begin{equation*}
d_{W}\left(\frac{E_{T,N}}{\sqrt{\psi^{\theta_0}_{N,T}}}, Z\right) \leq \frac{C(\theta_0)}{\sqrt{\psi^{\theta_0}_{N,T}}},
\end{equation*}
where $ Z \sim \mathcal{N}(0,1)$. On the other hand,
\begin{equation*}
\sqrt{\psi^{\theta_0}_{T,N}} (\theta_0 - \hat{\theta}_{T,N}) =
\frac{E_{T,N}/ \sqrt{\psi^{\theta_0}_{T,N}}}{<E_{T,N}>/
	\psi^{\theta_0}_{T,N}}.
\end{equation*}
Thus,
\begin{align*}
&d_W \left(\sqrt{\psi^{\theta_0}_{T,N}} (\theta_0 - \hat{\theta}_{T,N}), Z \right) = d_{W} \left(
\frac{E_{T,N}/ \sqrt{\psi^{\theta_0}_{T,N}}}{<E_{T,N}>/ \psi^{\theta_0}_{T,N}}, Z\right) \\
& \leq d_{W}\left( \frac{E_{T,N}}{\sqrt{\psi^{\theta_0}_{T,N}}}, Z\right) + d_W\left(\frac{E_{T,N}/ \sqrt{\psi^{\theta_0}_{T,N}}}{<E_{T,N}>/ \psi^{\theta_0}_{T,N}}, \frac{E_{T,N}}{\sqrt{\psi^{\theta_0}_{T,N}}} \right) \\
& \leq d_{W}\left( \frac{E_{T,N}}{\sqrt{\psi^{\theta_0}_{T,N}}}, Z\right) + \mathbf{E}\left[ \left|(\frac{E_{T,N}/ \sqrt{\psi^{\theta_0}_{T,N}}}{<E_{T,N}>/ \psi^{\theta_0}_{T,N}} - \frac{E_{T,N}}{\sqrt{\psi^{\theta_0}_{T,N}}} \right|\right] \\
& \leq d_{W}\left( \frac{E_{T,N}}{\sqrt{\psi^{\theta_0}_{T,N}}}, Z\right) + \|\frac{E_{T,N}/ \sqrt{\psi^{\theta_0}_{T,N}}}{<E_{T,N}>/ \psi^{\theta_0}_{T,N}} \|_{L^2} \times \|1- \frac{<E_{T,N}>} {\psi^{\theta_0}_{T,N}}\|_{L^2} \\
& \leq    d_{W}\left( \frac{E_{T,N}}{\sqrt{\psi^{\theta_0}_{T,N}}}, Z\right) + \| \frac{E_{T,N}}{\sqrt{\psi^{\theta_0}_{T,N}}} \|_{L^4} \times \|\frac{\psi^{\theta_0}_{T,N}}{<E_{T,N}>} \|_{L^4} \times \|1- \frac{<E_{T,N}>} {\psi^{\theta_0}_{T,N}} \|_{L^2}.
\end{align*}
The desired result follows from the previous calculus along with
Lemma \ref{CW} which implies the existence of a constant $C(\theta_0)$
such that $\|\frac{\psi^{\theta_0}_{T,N}}{<E_{T,N}>} \|_{L^4} < C(\theta_0)$.
The following corollary is a direct consequence of Theorem
\ref{main-thm}.
\end{proof}
\begin{corollary}\label{cor. of main-thm}
	Suppose that $\alpha > 0$, $\theta_0 >0$, $ \gamma \geq 0$, $d \geq 1$ in equation (\ref{SPDE}). Let $\hat{\theta}_{T,N}$ be the MLE given in (\ref{MLE}). Then, we have
	\begin{itemize}
		\item If $N \rightarrow +\infty$ and $T$ is fixed, then there exists a constant $C(\alpha, d, \theta_0,T)$ such that
		\begin{equation}
		d_{W} \left(N^{\frac{\alpha}{d} + \frac{1}{2}} (\theta_0- \hat{\theta}_{T,N}), \mathcal{N} \left(\frac{(4 \alpha/d +2)\theta_0}{\bar{\sigma}^\alpha T}\right) \right)     \leq C(\alpha, d, \theta_0,T) \times %
		\left\{
		\begin{array}{ll}
		\frac{1}{N^{\frac{2 \alpha}{d}}} & \mbox{ if } 2\alpha < d \\
		&  \\
		\frac{1}{N} & \mbox{ if } 2 \alpha =d%
		\\
		~~ &  \\
		\frac{1}{N^{\frac{\alpha}{d}+ \frac{1}{2}}} & \mbox{ if }2 \alpha > d. \\
		~~ &
		\end{array}%
		\right.
		\end{equation}
		\item If $T \rightarrow +\infty$ and $N$ is fixed, then there exists a constant $C(\alpha, d, \theta_0,N)$ such that
		\begin{equation}
		d_{W} \left(\sqrt{T} (\theta_0- \hat{\theta}_{T,N}), \mathcal{N}\left(0, \frac{2 \theta_0}{\sum \limits_{n=1}^{N} \ko} \right) \right)     \leq   \frac{C(\alpha, d, \theta_0,N)}{\sqrt{T}}.
		\end{equation}
		\item If $T \rightarrow + \infty$ and $N \rightarrow +\infty$, then there exists a  constant $C(\alpha, d, \theta_0) $ such that
		\begin{equation}
		d_{W} \left(\sqrt{T} N^{\frac{\alpha}{d} + \frac{1}{2}} (\theta_0- \hat{\theta}_{T,N}), \mathcal{N} \left(\frac{(4 \alpha/d +2)\theta_0}{\bar{\sigma}^\alpha}\right) \right)     \leq C(\alpha, d, \theta_0) \times %
		\left\{
		\begin{array}{ll}
		\frac{1}{\sqrt {T}} \frac{1}{N^{\frac{2 \alpha}{d}}} & \mbox{ if } 2\alpha < d \\
		&  \\
		\frac{1}{\sqrt{T}} \frac{1}{N} & \mbox{ if } 2 \alpha =d%
		\\
		~~ &  \\
		\frac{1}{\sqrt{T}} \frac{1}{N^{\frac{\alpha}{d}+ \frac{1}{2}}} & \mbox{ if }2 \alpha > d. \\
		~~ &
		\end{array}%
		\right.
		\end{equation}
	\end{itemize}
\end{corollary}
\begin{example}
	If the power of the Laplacian in equation (\ref{SPDE}) equals 1: $\alpha =1$, then equation (\ref{SPDE}) becomes the classical stochastic heat equation with an additive space-time white noise as follows :
	\begin{equation*}
	\left\{
	\begin{array}{ll}
	dY_\theta(t,x) - \theta \Delta Y_\theta(t,x) dt = \sum \limits_{n \in \mathbb{N}}
	\delta^{-\gamma}_n e_n(x) dw_n(t), \text{ \ } x \in G &  \\
	~~ &  \\
	Y_\theta(0,x) = 0,  \text{ \ } x \in G.
	\end{array}%
	\right.
	\end{equation*}
	Then, in this case, if $G \subset \mathbb{R}$, namely $d =1$, there exists a constant $C(\theta_0)$ depending only on $\theta_0$ such that
	the following bound of convergence in $d_W$ holds for the MLE
	$\hat{\theta}_{T,N}$ :
	\begin{equation*}
	d_{W} \left(\sqrt{T} N^{3/2}(\theta_0- \hat{\theta}_{T,N}),
	\mathcal{N}\left(0, \frac{6 \theta_0}{\bar{\sigma}} \right) \right)
	\leq   \frac{C(\theta_0)}{\sqrt{T}N^{3/2}}.
	\end{equation*}
\end{example}
\section{Approximate maximum likelihood estimator}
In this section, we study the approximative MLE of the parameter $\theta$, defined
by :
\begin{eqnarray}\label{tildetheta}
\widetilde{\theta}_{T,N,M}&:=&-\frac{\sum_{n=1}^N \delta_n^{2
		\alpha+2 \gamma} \sum_{i=1}^M
	y_n\left(t_{i-1}\right)\left[y_n\left(t_i\right)-y_n\left(t_{i-1}\right)\right]}{\Delta_M\sum_{n=1}^N
	\delta_n^{4 \alpha + 2 \gamma} \sum_{i=1}^M
	y_n^2\left(t_{i-1}\right)}\\&=&-\frac{\sum_{n=1}^N \delta_n^{2
		\alpha} \sum_{i=1}^M
	v_n\left(t_{i-1}\right)\left[v_n\left(t_i\right)-v_n\left(t_{i-1}\right)\right]}{\Delta_M\sum_{n=1}^N
	\delta_n^{4 \alpha} \sum_{i=1}^M v_n^2\left(t_{i-1}\right)},
\end{eqnarray}
which is a discretized version of the MLE given by \eqref{MLE},
where
$$y_n(t) = \delta_n^{-\gamma} \int_{0}^{t} e^{- \theta_0 \delta_n^{2 \alpha} (t-s)} dw_n(s),
\qquad v_n(t) =  \int_{0}^{t} e^{- \theta_0 \delta_n^{2 \alpha} (t-s)}
dw_n(s), \ n = 1,...,N.$$
In what
follows, we assume that the Fourier modes $y_n(t), n \geq 1$, are
observed at a uniform time grid:
$$
0=t_0<t_1<\cdots<t_M=T, \quad \text { with }
\Delta_M:=t_i-t_{i-1}=\frac{T}{M}, i=1, \ldots, M.
$$
We are interested in studying the consistency and the  rate of convergence of the
distribution of  $\widetilde{\theta}_{T,N,M}$ as $N, M, T
\rightarrow \infty$. For this aim, let us introduce the following sequences
\begin{eqnarray}
S_{N, M}:=\Delta_M\sum_{n=1}^N \delta_n^{4 \alpha} \sum_{i=1}^M
v_n^2\left(t_{i-1}\right), \ \ \text{and} \ \ \Lambda_{N, M}&:=&\sum_{n=1}^N \delta_n^{2 \alpha} \sum_{i=1}^M
e^{-\theta_0 \delta_n^{2 \alpha} t_{i}}
v_n\left(t_{i-1}\right)\left[\eta_n\left(t_i\right)-\eta_n\left(t_{i-1}\right)\right]\nonumber\\
&=&\sum_{n=1}^N \delta_n^{2 \alpha} \sum_{i=1}^Me^{-\theta_0
	\delta_n^{2 \alpha}(t_{i}+t_{i-1})}
\eta_n\left(t_{i-1}\right)\left[\eta_n\left(t_i\right)-\eta_n\left(t_{i-1}\right)\right]\nonumber\\
&=:&\sum_{n=1}^N B_{n,M},\label{Lambda-zeta}
\end{eqnarray} where $\eta_{n}(t)=\int_{0}^{t}e^{\theta_0 \delta_n^{2 \alpha}
	s}dw_n(s). \label{zeta}$ Thus,
\begin{eqnarray}
-\widetilde{\theta}_{T,N,M}=\frac{\sum_{n=1}^N
	\delta_n^{2\alpha}\left(e^{-\delta_n^{2 \alpha}\theta_0
		\Delta_{M}}-1\right)\sum_{i=1}^M
	v_n^2\left(t_{i-1}\right)}{\Delta_M\sum_{n=1}^N \delta_n^{4 \alpha}
	\sum_{i=1}^M v_n^2\left(t_{i-1}\right)}+\frac{\Lambda_{N, M}}{S_{N,
		M}}.\label{eq1.theta-tilde}
\end{eqnarray}
Hence
\begin{eqnarray}
\theta_0-\widetilde{\theta}_{T,N,M}&=&\frac{\sum_{n=1}^N
	\delta_n^{2\alpha}\left(e^{-\delta_n^{2 \alpha}\theta_0
		\Delta_{M}}-1+\delta_n^{2 \alpha}\theta_0
	\Delta_{M}\right)\sum_{i=1}^M
	v_n^2\left(t_{i-1}\right)}{\Delta_M\sum_{n=1}^N \delta_n^{4 \alpha}
	\sum_{i=1}^M v_n^2\left(t_{i-1}\right)}+\frac{\Lambda_{N, M}}{S_{N,
		M}}\label{eq2.theta-tilde}\\
&=:&\frac{R_{N, M}}{S_{N, M}}+\frac{\Lambda_{N, M}}{S_{N,
		M}}.\label{eq3.theta-tilde}
\end{eqnarray}
\subsection{Consistency of the estimator $\widetilde{\theta}_{T,N,M}$}
\begin{proposition}\label{consistency}
	Assume that
	\begin{equation}\label{hypothese}
	\Delta_M N^{\frac{2 \alpha}{d}} \rightarrow 0, \text{ \ } \text{as} \text{ \ } N,M,T \rightarrow + \infty.
	\end{equation}
	Then the estimator $\widetilde{\theta}_{T,N,M}$ is weakly consistent, namely the following convergence holds in probability
	\begin{equation*}
	\widetilde{\theta}_{T,N,M} \rightarrow \theta_0, \text{ \ }  \text{as} \text{ \ } N,M,T \rightarrow + \infty.
	\end{equation*}
\end{proposition}
\begin{proof}
	Following the decomposition (\ref{eq3.theta-tilde}), we have for any $\varepsilon >0$
	\begin{equation}
	\mathbf{P}\left(|\widetilde{\theta}_{T,N, M} - \theta_0 | > \varepsilon \right) \leq \mathbf{P}\left( \left|\frac{R_{N,M}}{S_{N,M}} \right| > \varepsilon/2  \right) + \mathbf{P}\left( \left|\frac{\Lambda_{N,M}}{S_{N,M}} \right| > \varepsilon/2  \right)
	\end{equation}
	Moreover, for any arbitrary fixed $\delta \in (0, \varepsilon/2)$, we have
	\begin{equation}\label{dec-conv-prob}
	\mathbf{P}\left( |\widetilde{\theta}_{T,N,M} - \theta_0| > \varepsilon \right) \leq \mathbf{P}\left( \left|\frac{R_{N,M}}{\psi^{\theta_0}_{T,N}}\right| > \delta  \right) + \mathbf{P}\left( \left|\frac{\Lambda_{N,M}}{\psi^{\theta_0}_{T,N}} \right| > \delta  \right)
	+ 2 \mathbf{P}\left( \left|\frac{S_{N,M}}{\psi^{\theta_0}_{T,N}} - 1 \right| > \frac{\varepsilon - 2\delta}{\varepsilon} \right)
	\end{equation}
	We recall that
	\begin{equation*}
	\psi^{\theta_0}_{T,N} = \frac{T}{2 \theta_0} \sum \limits_{n=1}^{N} \delta^{2 \alpha}_{n} \sim \frac{\bar{\sigma} d T N^{\frac{2 \alpha}{d}+1}}{(4 \alpha + 2d) \theta_0} \text{ \ } \text{as} \text{ \ } N,T \rightarrow +\infty.
	\end{equation*}
	For each term included in the inequality (\ref{dec-conv-prob}), the second moment need to be estimated by Markov's inequality. 
	\begin{lemma}\label{RNM}
		Consider the sequence $R_{N,M}$ defined in (\ref{eq3.theta-tilde}) and assume that the assumption (\ref{hypothese}) holds, then there exists a constant $C(\theta_0,\alpha, d, \bar{\sigma})$ such that for any $N,T,M$ large, we have
		\begin{equation*}
		\mathbf{E}\left[\left(\frac{R_{N,M}}{\psi^{\theta_0}_{T,N}} \right)^2\right] \leq
		C(\theta_0,\alpha, d, \bar{\sigma}) \Delta^2_M N^{4 \alpha /d}.
		\end{equation*}
	\end{lemma}
\begin{proof}
	Recall that the sequence $R_{N,M}$ defined in (\ref{eq3.theta-tilde}) can be written as follows
	\begin{equation*}
	R_{N,M} := \sum\limits_{n=1}^{N} s_n \sum\limits_{i=1}^{M} v^{2}_{n}(t_{i-1}),
	\end{equation*}
	where $$ s_{n} := \delta^{2\alpha}_{n} \left(e^{-\delta^{2\alpha}_n \theta_0 \Delta_{M}} - 1 + \delta^{2\alpha}_n \theta_0 \Delta_{M} \right), \text{ \ } n =1,...,N. $$
	\\
	Therefore,
	\begin{align*}
	\mathbf{E}\left[ R^2_{N,M} \right]  &=  \sum\limits_{n,l=1}^{N} s_n s_l \sum\limits_{i,j=1}^{M} \mathbf{E}\left[v^{2}_{n}(t_{i-1})v^{2}_{l}(t_{j-1})  \right] \\
	& =  \sum\limits_{\substack{n,l=1 \\ n \neq l}}^{N} s_n s_l \sum\limits_{i,j=1}^{M} \mathbf{E}\left[v^{2}_{n}(t_{i-1}) \right]
	\mathbf{E}\left[v^{2}_{l}(t_{j-1}) \right] +
	\sum\limits_{n=1}^{N} s^{2}_n \mathbf{E}\left[ \left( \sum\limits_{i=1}^{M}
	v^{2}_{n}(t_{i-1})\right)^{2}\right].
	\end{align*}
	Besides, using Isserlis' theorem, we obtain for any $n=1,..,N$,
	\begin{equation*}
	\mathbf{E}\left[ \left( \sum\limits_{i=1}^{M} v^{2}_{n}(t_{i-1})\right)^{2}\right] =
	\left(\sum\limits_{i=1}^{M} \mathbf{E}[v^{2}_{n}(t_{i-1})]\right)^2 + 2 \sum \limits_{i=1}^{M} \left(\mathbf{E}[v^{2}_{n}(t_{i-1})]\right)^2 + 2 \sum\limits_{\substack{i,j=1 \\ i \neq j}}^{M} \left( \mathbf{E}\left[v_{n}(t_{i-1})v_{n}(t_{j-1})  \right]\right)^{2}.
	\end{equation*}
	Therefore,
	\begin{align}\label{norm-RNM}
	\mathbf{E}\left[ R^2_{N,M} \right] & = \left[ \left(\sum\limits_{n=1}^{N} s_n \sum\limits_{i=1}^{M} \mathbf{E}[v^{2}_{n}(t_{i-1})]\right)^2 +
	2 \sum \limits_{n=1}^{N} s^{2}_n \sum\limits_{i,j=1}^{M} \left( \mathbf{E}[v_{n}(t_{i-1})v_{n}(t_{j-1})]\right)^2 \right] \nonumber \\
	& := A_{N,M,T} + B_{N,M,T}.
	\end{align}
	Since $\left|1-e^{-x}-x\right|/x^2\leq1$ for all $x>0$, by applying Mean
	Value Theorem twice, we have for any $n=1,...,N$
	$$s_n :=  \delta^{2 \alpha}_{n} \left[e^{-\delta^{2\alpha}_n \theta_0 \Delta_{M}} - 1 + \delta^{2\alpha}_n \theta_0 \Delta_{M} \right] \leq \theta_0^2 \delta^{6 \alpha}_{n} \Delta^2_{M}.  $$
	Besides, we also have the following estimate, for any $n=1,...,N$, we have
	$$\mathbf{E}[v^2_n(t)] = \int_{0}^{t} e^{-2 \theta_0 \delta^{2 \alpha}_{n} (t-s)} ds = \frac{1}{2 \theta_0 \delta^{2 \alpha}_{n}}(1-e^{-2 \theta_0  \delta^{2 \alpha}_{n} t }) \leq\frac{1}{2 \theta_0 \delta^{2 \alpha}_{n}}. $$
	Therefore, we can write
	\begin{align*}
	\frac{ 4 \theta_0^2 A_{N,M,T}}{T^2 \left(\sum\limits_{n=1}^{N} \delta^{2 \alpha}_{n}\right)^2} & \leq \frac{4 \theta_0^2}{T^2 \left(\sum\limits_{n=1}^{N} \delta^{2 \alpha}_{n}\right)^2} \left(
	\sum\limits_{n=1}^{N} \theta_0^2 \delta^{6 \alpha}_{n} \Delta^2_{M} \sum\limits_{i=1}^{M} \frac{1}{2 \theta_0 \delta^{2\alpha}_n}\right)^2  \leq \frac{M^2 \Delta^4_{M} \theta_0^4}{ T^2 \left(\sum\limits_{n=1}^{N} \delta^{2 \alpha}_{n}\right)^2} \left(\sum\limits_{n=1}^{N} \delta^{4 \alpha}_{n}\right)^2.
	\end{align*}
	Using the fact that $\delta_n \sim \bar{w}^{1/2} n^{1/d}$ as $n \rightarrow +\infty$, the following estimate holds for $A_{N,M,T}$ for $N,M,T$ large
	\begin{equation*}
	\frac{4 \theta_0^2 A_{N,M,T}}{{T^2 \left(\sum\limits_{n=1}^{N} \delta^{2 \alpha}_{n}\right)^2}} \leq {\theta_0^4} \left(\frac{2 \alpha+d}{\alpha+d}\right)^2
	\bar{\sigma}^{2 \alpha} \Delta^2_M N^{4\alpha/d}.
	\end{equation*}
	For the sequence $B_{N,M,T}$, we have using Cauchy-Schwarz inequality and the estimates above,
	
	\begin{align*}
	\frac{4 \theta_0^2 B_{N,M,T}}{{T^2 \left(\sum\limits_{n=1}^{N} \delta^{2 \alpha}_{n}\right)^2}}  \leq \frac{2}{T^2 \left(\sum\limits_{n=1}^{N} \delta^{2 \alpha}_{n}\right)^2} \sum\limits_{n=1}^{N} s^2_n \sum \limits_{i,j=1}^{M} \frac{1}{ \delta^{4 \alpha}_{n}}  \leq \frac{2 M^2 \Delta^4_M \theta_0^4}{ T^2 \left(\sum\limits_{n=1}^{N} \delta^{2 \alpha}_{n}\right)^2} \sum\limits_{n=1}^{N} \delta^{8 \alpha}_{n}
	\end{align*}
	as previsouly, using the fact that $\delta_n \sim \bar{\sigma}^{1/2} n^{1/d}$ as $n \rightarrow +\infty$, the following estimate holds for $B_{N,M,T}$ for $N,M,T$ large
	\begin{equation*}
	\frac{4 \theta_0^2 B_{N,M,T}}{{T^2 \left(\sum\limits_{n=1}^{N} \delta^{2 \alpha}_{n}\right)^2}} \leq  \left(\frac{2 \alpha+d}{8 \alpha+d}\right)^2 \bar{\sigma}^{2 \alpha} {2 \theta_0^4} \Delta^2_M N^{\frac{4 \alpha}{d}-1}.
	\end{equation*}
	The desired result follows.
\end{proof}
\begin{lemma}\label{norm-SNM}
	Consider the sequence $S_{N,M}$ defined in (\ref{eq3.theta-tilde}) and assume that assumption (\ref{hypothese}) holds true. Then, there exists a constant $C(\theta_0,\alpha,\bar{\sigma},d)$ such that for $N,M,T$ large we have
	\begin{equation*}
	\mathbf{E}\left[\left|\frac{S_{N,M}}{\psi^{\theta_0}_{T,N} } - 1 \right|^2 \right] \leq C(\theta_0,\alpha,\bar{\sigma},d) \max\left(\frac{1}{T N^{\frac{2 \alpha}{d} +1}}, \frac{1}{T^2 N^{\frac{4 \alpha}{d}}}\right).
	\end{equation*}
\end{lemma}
\begin{proof}
	Recall that	$ S_{N,M } = \Delta_{M} \sum\limits_{n=1}^{N} \delta^{4 \alpha}_{n} \sum\limits_{i=1}^{M} v^2_{n}(t_{i-1}).$
	Then, by the product formula (\ref{eq:product}), we obtain
	\begin{align*}
	 S_{N,M }  = \Delta_{M}\left[ \sum\limits_{n=1}^{N} \delta^{4 \alpha}_{n} \sum\limits_{i=1}^{M}
	I^{w_n}_2 \left(e^{-2 \theta_0 \delta^{2 \alpha}_n t_{i-1}}  e^{\theta_0  \delta^{2 \alpha}_{n} (u+v)} \mathbf{1}(u,v)_{[0,t_{i-1}]^2} \right) 	+ \sum\limits_{n=1}^{N} \delta^{4 \alpha}_{n} \sum\limits_{i=1}^{M} \frac{1}{2 \theta_0 \delta^{2 \alpha}_{n}}
	(1- e^{-2 \theta_0  \delta^{2 \alpha}_{n} t_{i-1} })  \right].
	\end{align*}
	Therefore, we can write
	\begin{equation}\label{Decomposition-SNM}
	\frac{S_{N,M}}{\psi^{\theta_0}_{T,N}} - 1 := C_{N,M,T} +
	\left(D_{N,M,T}-1 \right).
	\end{equation}
	Thus, using the inequality $(a+b)^2 \leq 2 a^2 + 2 b^2$, we obtain
	\begin{equation*}
	\mathbf{E}\left[\left|\frac{S_{N,M}}{\psi^{\theta_0}_{T,N}} - 1 \right|^2 \right] \leq 2 \mathbf{E}[C^2_{N,M,T}] + 2 |D_{N,M,T}-1|^2.
	\end{equation*}
	For the seek of simplicity, let us denote $\varphi^{\theta_0}_n(t) : = e^{-2 \theta_0 \delta^{2 \alpha}_n t }  e^{\theta_0  \delta^{2 \alpha}_{n}(u+v)} \mathbf{1}(u,v)_{[0,t]^2} $, we have
	\begin{align*}
	&\mathbf{E}\left[C^2_{N,M,T}\right]\\
	& = \frac{4 \theta_0^2 \Delta^2_{M}}{T^2 (\sum\limits_{n=1}^{N} \delta^{2 \alpha}_{n})^2}
	\left[\sum\limits_{n=1}^{N} \delta^{8 \alpha}_n \sum \limits_{i=1}^{M} \mathbf{E}\left[I^{w_n}_2(\varphi^{\theta_0}_n(t_{i-1}))^2 \right] + \sum\limits_{n=1}^{N} \delta^{8 \alpha}_n \sum \limits_{i,j=1}^{M}
	\mathbf{E}\left[I^{w_n}_2(\varphi^{\theta_0}_n(t_{i-1}))I^{w_n}_2(\varphi^{\theta_0}_n(t_{j-1})) \right] \right] \\
	& := C_{N,M,T,1} + C_{N,M,T,2}.
	\end{align*}
	For the sequence $C_{N,M,T,1}$, using the isometry property (\ref{Isometry}) the following estimate holds %and the fact that $(1-e^{-x}) \leq x$ for all $x> 0$, we have
	%\begin{align*}
	%C_{N,M,T,1} & := \frac{\Delta^2_{M}}{T^2 (\sum\limits_{k=1}^{N} \delta^{2 \alpha}_{k})^2} \sum\limits_{k=1}^{N} \delta^{8 \alpha}_k \sum\limits_{i=1}^{M} \frac{1}{2 \theta_0^2 \delta^{4 \alpha}_k}(1- e^{-2 \theta_0 \delta^{2 \alpha}_k t_{i-1}})^2 \\
	%& \leq \frac{2 \Delta^2_{M}}{T^2 (\sum\limits_{k=1}^{N} \delta^{2 \alpha}_{k})^2} \sum\limits_{k=1}^{N} \delta^{8 \alpha}_k \sum\limits_{i=1}^{M}t_{i-1}^2 \\
	%& \leq \frac{2 \Delta^4_{M}}{T^2 (\sum\limits_{k=1}^{N} \delta^{2 \alpha}_{k})^2} \sum\limits_{k=1}^{N} \delta^{8 \alpha}_k \sum\limits_{i=1}^{M}(i-1)^2 \\
	%& \leq   \frac{2 \Delta^4_{M}}{T^2 (\sum\limits_{k=1}^{N} \delta^{2 \alpha}_{k})^2} \frac{M(M+1)(2M+1)}{6} \sum\limits_{k=1}^{N} \delta^{8 \alpha}_{k}\sim c(\theta_0,\alpha, d,\bar{\sigma})\frac{T^2}{M}N^{\frac{4 \alpha}{d}-1}.
	%\end{align*}
	
	\begin{align*}
	C_{N,M,T,1} &= \frac{2 \Delta^2_{M}}{T^2 (\sum\limits_{n=1}^{N} \delta^{2 \alpha}_{n})^2} \sum\limits_{n=1}^{N} \delta^{8 \alpha}_n \sum\limits_{i=1}^{M} \frac{1}{ \delta^{4 \alpha}_n}(1- e^{-2 \theta_0 \delta^{2 \alpha t_{i-1}}_n})^2  \leq \frac{2 M\Delta^2_{M}}{T^2 (\sum\limits_{n=1}^{N} \delta^{2 \alpha}_{n})^2} \sum\limits_{n=1}^{N} \delta^{4 \alpha}_n \\
	& \leq c(\theta_0,\alpha, d,\bar{\sigma}) \frac{M \Delta^2_M}{T^2} \frac{N^{\frac{4 \alpha}{d} +1}}{ N^{\frac{4 \alpha}{d} +2}} \sim c(\theta_0,\alpha, d,\bar{\sigma}) \frac{1}{MN}.
	\end{align*}
	For the term $C_{N,M,T,2}$, we have
	\begin{align*}
	C_{N,M,T,2} & = \frac{4 \theta_0^2 \Delta^2_M}{T^2 (\sum\limits_{n=1}^{N} \delta^{2 \alpha}_{n})^2} \sum\limits_{n=1}^{N}  \delta^{8 \alpha}_{n} \sum \limits_{i,j=1}^{M} <\varphi^{\theta_0}_n(t_{i-1}), \varphi^{\theta_0}_n(t_{j-1})>_{L^{2}([0,T]^2)} \\
	& = \frac{8 \theta_0^2 \Delta^2_M}{T^2 (\sum\limits_{n=1}^{N} \delta^{2 \alpha}_{n})^2} \sum\limits_{n=1}^{N}  \delta^{8 \alpha}_{n}\sum \limits_{i=1}^{M-1} \sum\limits_{j=i+1}^{M} e^{-2 \theta_0 \delta^{2\alpha}_n(t_{j-1}-t_{i-1})} \left(\int_{0}^{t_{i-1}} e^{-2 \theta_0 \delta^{2 \alpha}_{n} u} du\right)^2   \\
	& = \frac{2 \Delta^2_M}{T^2 (\sum\limits_{n=1}^{N} \delta^{2 \alpha}_{n})^2} \sum\limits_{n=1}^{N} \delta^{4 \alpha}_{n} \sum \limits_{i=1}^{M-1} \sum\limits_{j=i+1}^{M}  e^{-2 \theta_0 \delta^{2\alpha}_n \Delta_M(j-i)} (1- e^{-2 \theta_0 \delta^{2 \theta_0}_n t_{i-1}})^{2} \\
	& \leq  \frac{2 \Delta^2_M}{ T^2 (\sum\limits_{n=1}^{N} \delta^{2 \alpha}_{n})^2} \sum\limits_{n=1}^{N} \delta^{4 \alpha}_{n} \sum \limits_{i=1}^{M-1} \sum\limits_{r=1}^{M-i}  e^{-2 \theta_0 \delta^{2\alpha}_n \Delta_M r} \\
	& \leq \frac{2 M \Delta^2_M}{ T^2 (\sum\limits_{n=1}^{N} \delta^{2 \alpha}_{n})^2}
	\sum\limits_{n=1}^{N} \delta^{4 \alpha}_{n} \left(\frac{1-e^{-2 \theta_0 \delta^{2 \alpha}_{n}T}}{1- e^{-2 \theta_0 \delta^{2 \alpha}_{n} \Delta_M}}\right) \\
	& \leq \frac{M \Delta_M}{\theta_0 T^2  (\sum\limits_{n=1}^{N} \delta^{2 \alpha}_{n})^2}  \sum\limits_{n=1}^{N} \delta^{2 \alpha}_{n} \frac{\Delta_M \delta^{2 \alpha}_{n} 2 \theta_0}{(1- e^{-2 \theta_0 \delta^{2 \alpha}_{n} \Delta_M})} \sim c(\theta_0,\alpha,d, \bar{w}) \frac{1}{T N^{\frac{2\alpha}{d}+1}},
	\end{align*}
	where we used the fact that under the assumption $\Delta_M N^{\frac{2 \alpha}{d}} \rightarrow 0$, as $N,M \rightarrow +\infty$, we have
	\begin{equation*}
	\delta^{2 \alpha}_{n}\frac{\Delta_M \delta^{2 \alpha}_{n} 2 \theta_0}{(1-e^{-2 \theta_0 \delta^{2 \alpha}_{n} \Delta_M})} \sim \delta^{2 \alpha}_{n}, \	\text{as} \ n,M \rightarrow +\infty.
	\end{equation*}
 Moreover, under the assumption
	$\Delta_M N^{\frac{2 \alpha}{d}} \rightarrow 0$, as $N,M,T
	\rightarrow +\infty$, we get $ \frac{1}{MN} = o\left(\frac{1}{T
		N^{\frac{2\alpha}{d}+1}}\right).$ \\
	For the left hand term
	$D_{N,M,T}$ in (\ref{Decomposition-SNM}), using similar arguments as
	previously, we have
	\begin{align*}
	D_{N,M,T} -1 & = \frac{\Delta_M}{T \sum\limits_{n=1}^{N} \delta^{2\alpha}_n} \sum\limits_{n=1}^{N} \delta^{2\alpha}_n
	\sum\limits_{i=1}^{M} (1-e^{-2 \theta_0 \delta^{2\alpha}_n t_{i-1}}) -1 \\
	& = \frac{1}{T \sum\limits_{n=1}^{N} \delta^{2\alpha}_n} \left[ \Delta_M \sum \limits_{n=1}^{N} \delta^{2\alpha}_n
	\sum\limits_{n=1}^{M} ((1-e^{-2 \theta_0 \delta^{2\alpha}_n t_{i-1}})-1 )\right] \\
	& = - \frac{\Delta_M}{T \sum\limits_{n=1}^{N} \delta^{2\alpha}_n} \sum\limits_{n=1}^{N} \delta^{2\alpha}_n \left(\frac{1-e^{-2 \theta_0 \delta^{2 \alpha}_{n}T}}{1- e^{-2 \theta_0 \delta^{2 \alpha}_{n} \Delta_M}}\right) .
	\end{align*}
	Hence
	\begin{equation*}
	|D_{N,M,T}-1|^2 \leq c(\theta_0,\alpha,d,\bar{\sigma}) \frac{1}{T^2 N^{4\alpha/d}}.
	\end{equation*}
\end{proof}
\begin{lemma}\label{estim-lambda}
	For any $N,M,T$, the following estimate holds for the sequence $\Lambda_{N,
		M}$,
	\begin{eqnarray}
	\left|\mathbf{E}\left[\left(\frac{\Lambda_{N,
			M}}{\sqrt{T\sum_{n=1}^{N}\delta_n^{2
				\alpha}}}\right)^2\right]-\frac{1}{2\theta_0}\right| &\leq&\Delta_{M}
	\frac{\sum_{n=1}^{N}\delta_n^{4 \alpha}}{\sum_{n=1}^{N}\delta_n^{2
			\alpha}} +\frac{N}{(2\theta_0)^2T\sum_{n=1}^{N}\delta_n^{2
			\alpha}}.\label{estimate-Lambda}
	\end{eqnarray}
	Moreover, if $N,M,T$ satisfy $\Delta_M N^{\frac{2\alpha}{d}} \rightarrow 0$, as $N,M,T \rightarrow +\infty$, then
	\begin{equation}\label{lim-norm-lambda}
	\mathbf{E}\left[\left(\frac{\Lambda_{N,
			M}}{\sqrt{T\sum_{n=1}^{N}\delta_n^{2
				\alpha}}}\right)^2\right] \longrightarrow \frac{1}{2 \theta_0},
	\end{equation}
	as $N,M,T \rightarrow +\infty.$
\end{lemma}
\begin{proof}
	Since $\eta_n,\ n=1,\ldots,N$ are Gaussian independent processes
	with independent increments, we can write
	\begin{equation*}
	\mathbf{E}\left[\left(\frac{\Lambda_{N,
			M}}{\sqrt{T}}\right)^2\right]=\frac{1}{T}\sum_{n=1}^N \mathbf{E}\left(B_{n,M}^2\right).
	\end{equation*}
	Next, since $\eta_{n}(t_{i-1})$ and
	$\eta_{n}(t_{i})-\eta_{n}(t_{i-1})$ are independent, we have
	\begin{align*}
	&\frac{1}{T}\mathbf{E}\left(B_{n,M}^2\right) \\
	&=\frac{\delta_n^{4
			\alpha}}{T}\sum_{i,j=1}^{M} e^{-\theta_0 \delta_n^{2
			\alpha}(t_{i}+t_{i-1}+t_{j}+t_{j-1})}
	\mathbf{E}\left[\eta_{n}(t_{i-1})\left(\eta_{n}(t_{i})-\eta_{n}(t_{i-1})\right)\eta_{n}(t_{j-1})\left(\eta_{n}(t_{j})-\eta_{n}(t_{j-1})\right)\right]\\
	&=\frac{\delta_n^{4 \alpha}}{T}\sum_{i=1}^{M} e^{-2\theta_0
		\delta_n^{2 \alpha}(t_{i}+t_{i-1})}
	\mathbf{E}\left[\eta_{n}(t_{i-1})^2\left(\eta_{n}(t_{i})-\eta_{n}(t_{i-1})\right)^2\right]
	\\
	&=\frac{\delta_n^{4 \alpha}}{T}\sum_{i=1}^{M} e^{-2\theta_0
		\delta_n^{2 \alpha}(t_{i}+t_{i-1})}
	\mathbf{E}\left[\eta_{n}(t_{i-1})^2\right]
	\mathbf{E}\left[\left(\eta_{n}(t_{i})-\eta_{n}(t_{i-1})\right)^2\right]\\
	&=\frac{\delta_n^{4 \alpha}}{T}\sum_{i=1}^{M} e^{-2\theta_0
		\delta_n^{2 \alpha}(t_{i}+t_{i-1})} \left(\frac{e^{2\theta_0
			\delta_n^{2 \alpha} t_{i-1}}-1}{2\theta_0 \delta_n^{2 \alpha}}\right)
	\left(\frac{e^{2\theta_0 \delta_n^{2 \alpha}  t_{i}}-e^{2\theta_0 \delta_n^{2 \alpha} t_{i-1}}}{2\theta_0 \delta_n^{2 \alpha}}\right)\\
	&=\frac{\left(1-e^{-2\theta_0 \delta_n^{2 \alpha}
			\Delta_{M}}\right)}{(2\theta_0)^2\Delta_{M}} \frac{1}{M}\sum_{i=1}^{M}
	\left(1-e^{-2\theta_0 \delta_n^{2 \alpha} t_{i-1}}\right)\\
	&=\frac{\left(1-e^{-2\theta_0 \delta_n^{2 \alpha}
			\Delta_{M}}\right)}{(2\theta_0)^2\Delta_{M}}
	-\frac{\left(1-e^{-2\theta_0 \delta_n^{2 \alpha}
			\Delta_{M}}\right)}{(2\theta_0)^2\Delta_{M}}\left(\frac{1-e^{-2\theta_0
			\delta_n^{2 \alpha} T}}{M(1-e^{-2\theta_0 \delta_n^{2 \alpha}
			\Delta_M})}\right)\\
	&=\frac{\left(1-e^{-2\theta_0 \delta_n^{2 \alpha}
			\Delta_{M}}\right)}{(2\theta_0)^2\Delta_{M}}
	-\frac{\left(1-e^{-2\theta_0 \delta_n^{2 \alpha}
			T}\right)}{(2\theta_0)^2T}.
	\end{align*}
	This implies,
	\begin{eqnarray*}
		\sum_{n=1}^N\left[\frac{1}{T}\mathbf{E}\left(B_{n,M}^2\right)-\frac{\delta_n^{2
				\alpha}}{2\theta_0}\right]&=&\sum_{n=1}^N\left[\frac{\left(1-e^{-2\theta_0
				\delta_n^{2 \alpha} \Delta_{M}}-2\theta_0 \delta_n^{2 \alpha}
			\Delta_{M}\right)}{(2\theta_0)^2\Delta_{M}} -\frac{\left(1-e^{-2\theta_0
				\delta_n^{2 \alpha} T}\right)}{(2\theta_0)^2T}\right]\\
		&\leq&\sum_{n=1}^N\left[\frac{\left|1-e^{-2\theta_0 \delta_n^{2
					\alpha} \Delta_{M}}-2\theta_0 \delta_n^{2 \alpha}
			\Delta_{M}\right|}{(2\theta_0)^2\Delta_{M}}
		+\frac{1}{(2\theta_0)^2T}\right]\\
		&\leq&\sum_{n=1}^N\left[ \delta_n^{4 \alpha}
		\Delta_{M}+\frac{1}{(2\theta_0)^2T}\right],
	\end{eqnarray*}
	where the last inequality follows from the fact that
	$\left|1-e^{-x}-x\right|/x^2\leq1$ for all $x>0$, by applying Mean
	Value Theorem twice.\\
	Therefore,
	\begin{eqnarray}
	\left|\frac{1}{T\sum_{n=1}^{N}\delta_n^{2
			\alpha}}\mathbf{E}\left(B_{n,M}^2\right)-\frac{1}{2\theta_0}\right|
	&\leq&\Delta_{M} \frac{\sum_{n=1}^{N}\delta_n^{4
			\alpha}}{\sum_{n=1}^{N}\delta_n^{2 \alpha}}
	+\frac{N}{(2\theta_0)^2T\sum_{n=1}^{N}\delta_n^{2
			\alpha}},\label{estimate-sum-B}
	\end{eqnarray}
	which proves \eqref{estimate-Lambda}. The limit (\ref{lim-norm-lambda}), is a direct consequence of (\ref{estimate-Lambda}) and the fact that when $N,M,T \rightarrow +\infty$, the following hold true
	\begin{itemize}
		\item $$ \Delta_{M} \frac{\sum_{n=1}^{N}\delta_n^{4
				\alpha}}{\sum_{n=1}^{N}\delta_n^{2 \alpha}} \sim \frac{T}{M} \bar{\sigma}^{\alpha} N^{\frac{2 \alpha}{d}}\left(\frac{2\alpha+d}{4 \alpha+d}\right). $$
		\item $$\frac{N}{(2\theta_0)^2 T \sum_{n=1}^{N} \delta^{2\alpha}_n} \sim \frac{(2\alpha+d)}{T N^{\frac{2\alpha}{d}} d (2\theta_0)^2 \bar{\sigma}^{\alpha}}. $$
	\end{itemize}
	The desired result follows.
\end{proof}
\end{proof}
\subsection{Asymptotic normality of the estimator $\widetilde{\theta}_{T,N,M}$}

\begin{lemma}\label{cum-lambda}
	We have
	\begin{eqnarray}
	\kappa_{3}\left(\frac{\Lambda_{N,
			M}}{\sqrt{T\sum_{n=1}^{N}\delta_n^{2 \alpha}}}\right) =0 
	  \ \text{and} \ \kappa_{4}\left(\frac{\Lambda_{N,
			M}}{\sqrt{T\sum_{n=1}^{N}\delta_n^{2 \alpha}}}\right)
	&\leq&\frac{1}{M\left(2\theta_0\right)^2}\frac{\sum_{n=1}^{N}\delta_n^{4
			\alpha}}{\left(\sum_{n=1}^{N}\delta_n^{2
			\alpha}\right)^2}\label{third-cumulant-Lambda}
	\end{eqnarray}
\end{lemma}
\begin{proof} Using $\mathbf{E}[\Lambda_{N,M}]=0$ and the fact that   $\eta_n,\ n=1,\ldots,N$ are Gaussian independent
	processes with independent increments, we can write
	\begin{eqnarray*}\kappa_{3}\left(\frac{1}{\sqrt{T\sum_{n=1}^{N}\delta_n^{2 \alpha}}}
		\Lambda_{N,M}\right)= \mathbf{E}\left[\left(\frac{1}{\sqrt{T\sum_{n=1}^{N}\delta_n^{2
					\alpha}}}
		\Lambda_{N,M}\right)^3\right]=\frac{1}{\left(T\sum_{n=1}^{N}\delta_n^{2
				\alpha}\right)^{3/2}}\sum_{n=1}^N \mathbf{E}\left(B_{n,M}^3\right).
	\end{eqnarray*}
	Moreover, since  $\eta_n({t_{i-1}})$ and
	$\eta_n({t_{i}})-\eta_n({t_{i-1}})$ are independent, then
	$\mathbf{E}\left(B_{n,M}^3\right)=0$, which implies
	\eqref{third-cumulant-Lambda}.\\
	Since  $B_{n,M},\ n=1,\ldots,N$ are independent and $\mathbf{E}\left(B_{n,M}\right)=0$,
	\begin{eqnarray*} \mathbf{E}\left(\Lambda_{N, M}^4\right)&=&3\sum_{n\neq j=1}^N \mathbf{E}\left(B_{n,M}^2\right) \mathbf{E}\left(B_{j,M}^2\right)+\sum_{n=1}^N
		\mathbf{E}\left(B_{n,M}^4\right)
	\end{eqnarray*}
	and	$\mathbf{E}\left[\Lambda_{N,
			M}^2\right]=\sum_{n=1}^N \mathbf{E}\left(B_{n,M}^2\right).$ Hence
	\begin{eqnarray*}\kappa_{4}\left(\Lambda_{N, M}\right)&=& \mathbf{E}\left[\Lambda_{N, M}^4\right]-3\left[\mathbf{E}\left(\Lambda_{N,
			M}^2\right)\right]^2\\
		&=&\sum_{k=1}^N
		\mathbf{E}\left(B_{n,M}^4\right)-3\sum_{n=1}^N\left[\mathbf{E}\left(B_{n,M}^2\right)\right]^2\\
		&=&\sum_{n=1}^N\left(
		\mathbf{E}\left(B_{n,M}^4\right)-3\left[\mathbf{E}\left(B_{n,M}^2\right)\right]^2\right).
	\end{eqnarray*}
	On the other hand, similar arguments as in \cite{EAA2023} imply
	\begin{eqnarray*}
		\left|\mathbf{E}\left(B_{n,M}^4\right)-3\left[\mathbf{E}\left(B_{n,M}^2\right)\right]^2\right|
		&\leq&\left(\delta_n^{2 \alpha}\right)^4M\frac{\left(1-e^{-2\theta_0
				\delta_n^{2 \alpha}
				\Delta_{M}}\right)^2}{\left(2\theta_0 \delta_n^{2 \alpha}\right)^4} \leq \left(\delta_n^{2 \alpha}\right)^2M\frac{
			\Delta_{M}^2}{\left(2\theta_0 \right)^2},
	\end{eqnarray*}
	where the last inequality follows from the fact that
	$\left|1-e^{-x}\right|/x\leq1$ for all $x>0$, by   Mean Value
	Theorem.\\
	Therefore,
	\begin{eqnarray*}
		\kappa_{4}\left(\frac{\Lambda_{N,
				M}}{\sqrt{T\sum_{n=1}^{N}\delta_n^{2 \alpha}}}\right)
		&\leq&\frac{1}{M\left(2\theta_0\right)^2}\frac{\sum_{n=1}^{N}\delta_n^{4
				\alpha}}{\left(\sum_{n=1}^{N}\delta_n^{2 \alpha}\right)^2},
	\end{eqnarray*}
	which implies \eqref{third-cumulant-Lambda}
\end{proof}
\begin{theorem}\label{TCL-lambda}
	There exists a constant $C(\alpha, \theta_0, d)$ such that for $N,M,T$ large, we have
	\begin{equation}\label{bound-TCL-lambda}
	d_W\left(\frac{\Lambda_{N,
			M}}{\sqrt{\psi^{\theta_0}_{T,N}}}, Z\right) \leq C(\alpha, \theta_0, d) \times \frac{1}{MN}.
	\end{equation}
	where $\psi^{\theta_0}_{T,N} = \frac{T}{2 \theta_0} \sum\limits_{n=1}^{N}\delta_n^{2 \alpha}$ and $Z \sim \mathcal{N}(0,1)$. Consequently, as $N,M,T \rightarrow +\infty.$
	\begin{equation*}
	\frac{\Lambda_{N,M}}{\sqrt{T} N^{\frac{\alpha}{d}+ \frac{1}{2}}} \overset{law}{\longrightarrow}
	\mathcal{N}\left(0, \frac{\bar{\sigma}^{\alpha}d}{\theta_0(4 \alpha+2d)}\right).
	\end{equation*}
\end{theorem}
\begin{proof}
	Theorem (\ref{TCL-lambda}) is a direct consequence of (\ref{optimal berry esseen}), Lemma (\ref{cum-lambda}) and the fact that for $N,M$ large, we have
	$$ \frac{1}{M}\frac{\sum_{n=1}^{N}\delta_n^{4
			\alpha}}{\left(\sum_{n=1}^{N}\delta_n^{2 \alpha}\right)^2}\sim \frac{1}{MN} \frac{1}{d} \frac{(2 \alpha+d)^2}{4\alpha+d}.$$
\end{proof}
By the decomposition (\ref{eq3.theta-tilde}), we can write
\begin{equation*}
\sqrt{\psi^{\theta_0}_{T,N}} \left(\theta_0 - \tilde{\theta} _{T,N,M}\right) = \sqrt{\psi^{\theta_0}_{T,N}}  \frac{R_{N,M}}{S_{N,M}} + \frac{\Lambda_{N,
		M}/\sqrt{\psi^{\theta_0}_{T,N}}}{S_{N,M}/ \psi^{\theta_0}_{T,N}}.
\end{equation*}
Therefore, we can write using the fact that $d_W\left(X+Y, Z\right) \leq E[|X|] + d_W\left(Y,Z\right) $ for all $X,Y,Z$ random variables such that $X \in L^1(\Omega)$.
\begin{align}\label{dwtheta}
d_W\left( \sqrt{\psi^{\theta_0}_{T,N}} \left(\theta_0 - \tilde{\theta} _{T,N,M}\right), Z\right) & \leq \sqrt{\psi^{\theta_0}_{T,N}}  \mathbf{E}\left[\left|\frac{R_{N,M}}{S_{N,M}} \right| \right] + d_W\left(\frac{\Lambda_{N,
		M}/\sqrt{\psi^{\theta_0}_{T,N}}}{S_{N,M}/ \psi^{\theta_0}_{T,N}} ,
Z\right).
\end{align}
On the other hand, using triangular inequality for the distance $d_W$ we get
\begin{align*}
d_W\left(\frac{\Lambda_{N,
		M}/\sqrt{\psi^{\theta_0}_{T,N}}}{S_{N,M}/ \psi^{\theta_0}_{T,N}} , Z\right) &
\leq d_W\left(\frac{\Lambda_{N,
		M}}{\sqrt{\psi^{\theta_0}_{T,N}}}, Z\right) + d_W\left( \frac{\Lambda_{N,
		M}/\sqrt{\psi^{\theta_0}_{T,N}}}{S_{N,M}/ \psi^{\theta_0}_{T,N}}, \frac{\Lambda_{N,
		M}}{\sqrt{\psi^{\theta_0}_{T,N}}}\right) \\
& \leq d_W\left(\frac{\Lambda_{N,
		M}}{\sqrt{\psi^{\theta_0}_{T,N}}}, Z\right) + \mathbf{E}\left[\left|\frac{\Lambda_{N,
		M}/\sqrt{\psi^{\theta_0}_{T,N}}}{S_{N,M}/ \psi^{\theta_0}_{T,N}} - \frac{\Lambda_{N,
		M}}{\sqrt{\psi^{\theta_0}_{T,N}}}\right|\right] \\
& \leq d_W\left(\frac{\Lambda_{N,
		M}}{\sqrt{\psi^{\theta_0}_{T,N}}}, Z\right) + \|\frac{\Lambda_{N,
		M}/\sqrt{\psi^{\theta_0}_{T,N}}}{S_{N,M}/ \psi^{\theta_0}_{T,N}}\|_{L^2} \times \|1-  \frac{S_{N,M}}{ \psi^{\theta_0}_{T,N}}\|_{L^2} \\
& \leq d_W\left(\frac{\Lambda_{N,
		M}}{\sqrt{\psi^{\theta_0}_{T,N}}}, Z\right) + \|\frac{\Lambda_{N,
		M}}{\sqrt{\psi^{\theta_0}_{T,N}}}\|_{L^4} \times \| \frac{\psi^{\theta_0}_{T,N}}{S_{N,T}}\|_{L^4} \times   \|1-  \frac{S_{N,M}}{ \psi^{\theta_0}_{T,N}}\|_{L^2}.
\end{align*}
Notice that the sequence $\Lambda_{N,
	M}$ is a second chaos random variable, indeed by (\ref{Lambda-zeta}), we can write
\begin{align*}
\Lambda_{N,
	M} & = \sum\limits_{n=1}^{N} \delta^{2\alpha}_n\sum\limits_{i=1}^{M} I^{w_n}_1\left(e^{-\theta_0 \delta^{2\alpha}_n(t_{i-1}-.)} \mathbf{1}(
.)_{[0,t_{i-1}]} \right) I^{w_n}_1\left(e^{-\theta_0 \delta^{2\alpha}_n(t_{i}-.)} \mathbf{1}(
.)_{[t_{i-1},t_i]}\right) \\
& = \sum\limits_{n=1}^{N} I^{w_n}_2 \left[\delta^{2\alpha}_n \sum\limits_{i=1}^{M} \left(e^{-\theta_0 \delta^{2\alpha}_n(t_{i-1}-.)} \mathbf{1}(
.)_{[0,t_{i-1}]} \otimes e^{-\theta_0 \delta^{2\alpha}_n(t_{i}-.)} \mathbf{1}(
.)_{[t_{i-1},t_i]} \right)\right].
\end{align*}
Then, by the hypercontractivity property (\ref{hypercontractivity}), we get $\|\frac{\Lambda_{N,
		M}}{\sqrt{\psi^{\theta_0}_{T,N}}}\|_{L^4} \leq 3 \|\frac{\Lambda_{N,
		M}}{\sqrt{\psi^{\theta_0}_{T,N}}}\|_{L^2} < +\infty $.\\ On the other hand, by (\ref{disc-case}) in Lemma \ref{CW} , we have $ \| \frac{\psi^{\theta_0}_{T,N}}{S_{N,T}}\|_{L^4}  < +\infty.$ Finally, by Lemma \ref{norm-SNM}, the following estimates hold true
\begin{equation}\label{bound-norm-SNM}
\| \frac{S_{N,M}}{\psi^{\theta_0}_{T,N}} -1 \|_{L^2}    \leq C(\alpha, \theta_0, \bar{\sigma}, d) \times b_{N,T}%
\end{equation}
with $b_{N,T}$ is the same bound in (\ref{crochetENT}). A comparison, between the bound (\ref{bound-TCL-lambda}) obtained in Lemma \ref{TCL-lambda} and (\ref{crochetENT}), we can easily checn that under the assumption $\Delta_M N^{\frac{2\alpha}{d}} \rightarrow 0$ as $N,M,T \rightarrow +\infty$, we get
\begin{equation}\label{bound-norm-SNM}
\frac{1}{MN} = \left\{
\begin{array}{ll}
o\left(\frac{1}{\sqrt {T}} \frac{1}{N^{\frac{2 \alpha}{d}}}\right) & \mbox{ if } 2\alpha < d \\
&  \\
o\left(\frac{1}{\sqrt{T}} \frac{1}{N}\right) & \mbox{ if } 2 \alpha =d%
\\
~~ &  \\
o\left(\frac{1}{\sqrt{T}} \frac{1}{N^{\frac{\alpha}{d}+ \frac{1}{2}}}\right) & \mbox{ if }2 \alpha > d. \\
~~ &
\end{array}%
\right.
\end{equation}
Now, by (\ref{dwtheta}) it remains to control the convergence to zero of the term $\sqrt{\psi^{\theta_0}_{T,N}} \mathbf{E}\left[\left|\frac{R_{N,M}}{S_{N,M}}\right|\right]$ when $T,M,N \rightarrow + \infty$. Indeed, by Cauchy-Schwartz inequality, we have
\begin{equation*}
\sqrt{\psi^{\theta_0}_{T,N}} \mathbf{E}\left[\left|\frac{R_{N,M}}{S_{N,M}}\right|\right] \leq \mathbf{E}\left[\left(\frac{R_{N,M}}{\sqrt{\psi^{\theta_0}_{T,N}}}\right)^{2}\right]^{1/2} \times \mathbf{E}\left[\left(\frac{\psi^{\theta_0}_{T,N}}{S_{N,M}}\right)^{2}\right]^{1/2}.
\end{equation*}
By by (\ref{disc-case}) in Lemma \ref{CW} , we have $ \| \frac{\psi^{\theta_0}_{T,N}}{S_{N,T}}\|_{L^2}  < +\infty.$ On the other hand, from  (\ref{norm-RNM}) of Lemma \ref{RNM}, we have $\mathbf{E}[R^2_{N,M}] = A_{N,M,T} + B_{N,M,T}$. Moreover,
\begin{equation*}
\frac{A_{N,M,T}}{\psi^{\theta_0}_{T,N}} \leq \frac{\theta_0^4 M^2 \Delta^4_M}{ T \left(\sum\limits_{n=1}^{N}\delta^{2\alpha}_n \right)}
\left( \sum\limits_{n=1}^{N}\delta^{4\alpha}_n \right)^2  \leq c(\theta_0,\alpha, \bar{\sigma}, d) \frac{T^3}{M^2}
N^{\frac{6\alpha}{d}+1},
\end{equation*}
and
\begin{equation*}
\frac{B_{N,M,T}}{\psi^{\theta_0}_{T,N}} \leq \frac{2\theta_0^4 M^2 \Delta^4_M}{ T \left(\sum\limits_{n=1}^{N}\delta^{2\alpha}_n \right)}  \sum\limits_{n=1}^{N}\delta^{8\alpha}_n \leq
c(\theta_0,\alpha, \bar{\sigma}, d) \frac{T^3}{M^2} N^{\frac{6\alpha}{d}}.
\end{equation*}
We conclude that if
\begin{equation*}
\frac{T^{3/2} N^{\frac{3 \alpha}{d}+ \frac{1}{2}}}{M} \rightarrow 0 \text{ \ } \text{as} \text{ \ } {N,M,T \rightarrow +\infty.}
\end{equation*}
Then, there exists a constant $c(\theta_0,\alpha, \bar{\sigma}, d) > 0$ such that
\begin{equation}
\sqrt{\psi^{\theta_0}_{T,N}} \mathbf{E}\left[\left|\frac{R_{N,M}}{S_{N,M}}\right|\right] \leq c(\theta_0,\alpha, \bar{\sigma}, d) \frac{T^{3/2} N^{\frac{3 \alpha}{d}+ \frac{1}{2}}}{M} \rightarrow 0
\end{equation}
as $M,N,T \rightarrow + \infty.$
\begin{theorem}\label{thm-theta-tilde}
	Consider the estimator $\tilde{\theta}_{T,N,M}$ defined in
	(\ref{tildetheta}). Then, there exists a constant $C(\alpha,
	\theta_0,\bar{\sigma}, d)$ such that
	\begin{equation*}
	d_W\left( \sqrt{\psi^{\theta_0}_{T,N}} \left(\theta_0 - \tilde{\theta} _{T,N,M}\right), Z \right) \leq C(\alpha, \theta_0,\bar{\sigma}, d) \max\left( \frac{T^{3/2} N^{\frac{3 \alpha}{d}+ \frac{1}{2}}}{M}, b_{N,T} \right)
	\end{equation*}
	where $b_{N,T}$ is defined in (\ref{crochetENT}) and $Z \sim \mathcal{N}(0,1)$. In particular, if
	\begin{equation*}
	\frac{T^{3/2} N^{\frac{3 \alpha}{d}+ \frac{1}{2}}}{M} \rightarrow 0 \text{ \ } \text{as} \text{ \ } {N,M,T \rightarrow +\infty.}
	\end{equation*}
	Then,
	as $T,M,N \rightarrow +\infty$,
	\begin{equation*}
	\sqrt{T} N^{\frac{\alpha}{d}+\frac{1}{2}}\left(\theta_0 - \tilde{\theta}_{T,N,M}\right)\overset{law}{\longrightarrow }\mathcal{N}%
	\left(0, \frac{(4\alpha /d +2) \theta_0}{\bar{\sigma}^{\alpha}} \right).
	\end{equation*}
\end{theorem}
\section{Appendix}

\begin{lemma}\label{CW}
	Let $N \geq 1$ and consider $\psi^{\theta_0}_{T,N}$ defined in (\ref{ENT}), then for every $p, T_0 > 0$, we have
	\begin{eqnarray}\label{cont-case}
	\sup_{T\geq T_0}
	\mathbf{E}\left[\left(\frac{1}{\psi^{\theta_0}_{T,N}}\sum_{i=1}^{N}
	\delta_i^{4\alpha}\int_0^T v_i^2(s)
	ds\right)^{-p}\right]<\infty,\label{average-cont}
	\end{eqnarray}
	Moreover, for every $p>0$, there exists $M_0 \geq 1$ such that
	\begin{eqnarray}\label{disc-case}
	\sup_{M\geq M_0}
	\mathbf{E}\left[\left(\frac{1}{T \sum\limits_{j=1}^{N} \delta^{2 \alpha}_j} \Delta_M \sum_{i=1}^{N}
	\delta_i^{4\alpha} \sum\limits_{k=0}^{M-1} v^{2}_i(t_k)
	\right)^{-p}\right] < +\infty,\label{average-discret}
	\end{eqnarray}
\end{lemma}
\begin{proof}
The proof of assertion (\ref{cont-case}) relies on similar arguments as in the proof of Lemma 1 in \cite{EAA}. For the discrete case (\ref{disc-case}) Let $T>0$. Recall that for $k=1,...,N$, the processes $v_k(t):= \int_{0}^{t} e^{-\theta_0 \delta^{2\alpha}_k(t-u)} dw_k(u)$, $t \in [0,T]$, where $w_1,...,w_n$ are $N$ independent Brownian motions. Let $w=(w_1,..,w_N)$ and define ${\mathcal{F}_t}^w= \sigma(\bigcup_{i=1}^{N} {\mathcal{F}_t}^{w_i})$ where
$\mathcal{F}_t^{w_i}=\sigma\{w_i(u),u\leq t\}$  denotes the sigma-field generated by $w_i$ for $i=1,...,N$. \\
Fix $N \geq 1$ and $p>0$ and let $M_0 \geq 1$ an integer such that for every $M \geq M_0$ and $ M > 2p$, we have
\begin{equation}\label{estim-norm-OU}
\int_{0}^{\frac{T}{M}} e^{-2 \theta_0 \delta^{2 \alpha}_i u} du \geq \frac{1}{4 \theta_0 \delta^{2 \alpha}_i}, \text{ \ } i=1,...,N.
\end{equation}
Then, consider the following processes defined for every $i=1,...,N$ as
\begin{equation*}
Y_{i,t} := \sum\limits_{k=0}^{M-1} v_i(t_k)
\textbf{1}_{]t_{k-1},t_k]}(t),
\end{equation*}
where for $k=0,...,M$ $t_k = k \Delta_M$, $\Delta_M = \frac{T}{M} \rightarrow 0$ and $ M\Delta_M \rightarrow +\infty$ as $M \rightarrow +\infty$.
We can easily check that
\begin{equation*}
\frac{1}{T} \int_{0}^{T} Y^2_{i,t} dt = \frac{1}{M}
\sum\limits_{k=0}^{M-1} v^{2}_i(t_k).
\end{equation*}
We can follow the same steps of the continuous case with $Y_{i,t}$ instead of $v_i(t)$. The only remaining part to be proved is a lower bound for all $M \geq M_0$ of the term for any $k=1,...,N$
\begin{equation}
A_{M,N,k} : = \mathbf{E}\left[ \frac{\sum\limits_{i= 1}^{N}
	\delta^{4 \alpha}_i}{T \sum\limits_{j= 1}^{N} \delta^{4 \alpha}_j}
\int_{(k-1)T/M}^{kT/M} Y^2_{i,t} dt \mid\mathcal{F}^w_{(k-1)T/M}
\right].
\end{equation}
Or for $k=1,...,N$, we have for any $M \geq M_0$,
\begin{align*}
A_{M,N,k}& = \frac{1}{M} \frac{\sum\limits_{i=1}^{N} \delta^{4 \alpha}_i}{\sum \limits_{j=1}^{N} \delta^{2 \alpha}_j} \int_{(k-1)T/M}^{kT/M} \mathbf{E}[v^2_i(t_k) \mid\mathcal{F}^{w_i}_{(k-1)T/M}] \geq \frac{1}{M} \frac{\sum\limits_{i=1}^{N} \delta^{4 \alpha}_i}{\sum \limits_{j=1}^{N} \delta^{2 \alpha}_j} \left[  \int_{(k-1)T/M}^{kT/M} e^{-2 \theta_0 \delta^{2 \alpha}_i( \frac{kT}{M} -u)} du\right]\\
&  \geq \frac{1}{M} \frac{\sum\limits_{i=1}^{N} \delta^{4 \alpha}_i}{\sum \limits_{j=1}^{N} \delta^{2 \alpha}_j} \int_{0}^{T/M} e^{-2 \theta_0 \delta^{2 \alpha}_i u} du \geq \frac{1}{M} \frac{\sum\limits_{i=1}^{N} \delta^{4 \alpha}_i}{\sum \limits_{j=1}^{N} \delta^{2 \alpha}_j} \frac{1}{4 \theta_0 \delta^{2\alpha}_i} = \frac{1}{4 \theta_0 M}.
\end{align*}
which yields the proof of (\ref{average-discret}).
\end{proof}
\section{Acknowlegment}
This research project was funded by Princess Nourah bint
Abdulrahman University Researchers Supporting Project number
(PNURSP2024R358), Princess Nourah bint Abdulrahman University,
Riyadh, Saudi Arabia.

\end{document}